%% file: UpdatedArXiV.tex
\documentclass[12pt]{article}
\usepackage{fullpage}

\usepackage[latin1]{inputenc}

\usepackage{latexsym}
\usepackage{amssymb}
\usepackage{amsmath}
\usepackage{amsthm}
\usepackage{amscd}
\usepackage{graphicx}
\usepackage{transparent}
\usepackage{esint}
\usepackage{color}
\usepackage{IEEEtrantools}
\usepackage{subcaption}
\usepackage{url}

\theoremstyle{plain}
\newtheorem{theorem}{Theorem}
\newtheorem{lemma}[theorem]{Lemma}
\newtheorem{corollary}[theorem]{Corollary}

\theoremstyle{definition}

\theoremstyle{remark}

\newcommand\blfootnote[1]{%
  \begingroup
  \renewcommand\thefootnote{}\footnote{#1}%
  \addtocounter{footnote}{-1}%
  \endgroup
}

\def\D{\mathrm{d}}
\DeclareMathOperator{\re}{Re}
\DeclareMathOperator{\II}{Im}

\title{\bf Asymptotics of Bivariate Analytic Functions\\ with Algebraic Singularities}

\author{Torin Greenwood \thanks{Supported in part by NSF grant DMS-1344199.}\\
\small School of Mathematics\\ [-0.8ex]
\small Georgia Institute of Technology\\[-0.8ex]
\small Atlanta, GA, U.S.A.\\
\small \tt greenwood@math.gatech.edu\\
}

\date{June 20, 2017}

\begin{document}
\maketitle
\begin{abstract}
In this paper, we use the multivariate analytic techniques of Pemantle and Wilson to derive asymptotic formulae for the coefficients of a broad class of multivariate generating functions with algebraic singularities.  Then, we apply these results to a generating function encoding information about the stationary distributions of a graph coloring algorithm studied by Butler, Chung, Cummings, and Graham (2015).  Historically, Flajolet and Odlyzko (1990) analyzed the coefficients of a class of univariate generating functions with algebraic singularities.  These results have been extended to classes of multivariate generating functions by Gao and Richmond (1992) and Hwang (1996, 1998), in both cases by immediately reducing the multivariate case to the univariate case.  Pemantle and Wilson (2013) outlined new multivariate analytic techniques and used them to analyze the coefficients of rational generating functions.  These multivariate techniques are used here to analyze functions with algebraic singularities.

\end{abstract}

\blfootnote{\copyright\,2017.  This manuscript version is made available under the CC-BY-NC-ND 4.0 license 

\ \ \ \url{http://creativecommons.org/licenses/by-nc-nd/4.0/}}

\section{Introduction}

{For several decades, singularity analysis has been used to derive asymptotic formulae for coefficients of univariate generating functions.  In 1990, for example, Flajolet and Odlyzko found asymptotics for a large class of univariate functions with algebraic singularities {in \cite{FO1990}}.  Examining the coefficients of multivariate generating functions is notoriously more difficult and technical.  Pemantle and Wilson developed techniques to tackle multivariate rational generating functions in \cite{PW2013} and previous work, where they rely on the multivariate Cauchy integral, identifying and analyzing critical regions in the domain of integration that contribute to the integral's asymptotics through Morse theory.}  In this paper, we will look at the coefficients of $H(x, y)^{-\beta}$, where $H$ is an analytic function and $\beta \not \in \mathbb{Z}_{\leq 0}$ is a real number.  Under some assumptions about the zero set of $H$, we will find an asymptotic approximation for the coefficients $[x^ry^s]H(x, y)^{-\beta}$ as $r$ and $s$ approach infinity with $\frac{r}{s}$ in a nearly-fixed ratio, as described in Theorem \ref{MainResult}.

  Flajolet and Odlyzko's 1990 results relied on using the Cauchy integral formula and explicit contour manipulations.  {Later in the 1990s, Gao, Richmond, Bender, and Hwang} extended these results to classes of bivariate functions by temporarily fixing a variable and applying univariate results, which required special restrictions on the bivariate functions.  {(See Section \ref{history} below for more details.)}  In this paper, we instead rely on the multivariate techniques that Pemantle and Wilson developed, manipulating the multivariate Cauchy integral formula directly.  More details {of these techniques} are in Section \ref{MAC} below.  {By using a combination of the Pemantle and Wilson techniques and the contour manipulations of the original Flajolet and Odlyzko work, we avoid using Morse theory.  The algebraic singularities lead to manipulations of the torus on a Riemann surface instead of multidimensional complex space.  However, this does not change the main methods of the asymptotic analysis, except requiring careful tracking of the argument of some expressions.}
    
    {In Section \ref{Bivariate}, we state our main result (Theorem \ref{MainResult}), which we prove in subsequent sections.}  Then, in Section \ref{Examples}, we look at examples of our results, including an application of Theorem \ref{MainResult} to a generating function that encodes properties of the stationary distributions of random colorings on the complete graphs, as found in \cite{GG2015}.
    
    An extended abstract of this paper, \cite{Gre2016}, {appeared} in the proceedings of the 28th International Conference on Formal Power Series and Algebraic Combinatorics.

{\section{Historical Background}
In this section, we provide some information about previous results in singularity analysis on which this work relies.}

\subsection{Multivariate Analytic Combinatorics of Rational Functions}
\label{MAC}

In \cite{PW2013}, Pemantle and Wilson outline a program which greatly extends the results of previous work on multivariate generating function analysis.  {Although many of the technical details of the program are not needed to prove the results in this paper, Pemantle and Wilson's work still lays the foundation for our approach.}  In the simplest case, Pemantle and Wilson begin with a rational function, $F(\mathbf{z}) = G(\mathbf{z})/H(\mathbf{z})$, where $G$ and $H$ are polynomials with real coefficients in the variables $z_1, \ldots, z_d$, and where $F(\mathbf{z})$ is analytic near the origin.  We write $\mathbf{z} = (z_1, \ldots, z_d)$ and $\mathbf{z}^{\mathbf{r}} = z_1^{r_1} \cdots z_d^{r_d}$. Then, $F(\mathbf{z})$ has the series representation:
\[
F(\mathbf{z}) = \sum_{\mathbf{r} \in \mathbb{N}^d} a_{\mathbf{r}} \mathbf{z}^{\mathbf{r}}.
\]
  The multivariate Cauchy integral formula tells us {for $\mathbf{r} \in \mathbb{Z}^d_{\geq 0}$}:
\begin{equation}
 \left[\mathbf{z}^\mathbf{r}\right] F(\mathbf{z}) = \left(\frac{1}{2 \pi i}\right)^d \int_{T} F(\mathbf{z}) \mathbf{z}^{-\mathbf{r}-1} \, \D \mathbf{z}. \label{MCauchy}
 \end{equation}
Here, the torus $T = \{|z_1| = c_1\} \times \cdots \times \{|z_d| = c_d\}$ is small enough that it does not enclose any singularities of $F(\mathbf{z})$.  The goal is to approximate $\left[ \mathbf{z}^{n \hat{\mathbf{r}}}\right] F(\mathbf{z})$ for some fixed unit vector $\hat{\mathbf{r}} \in \mathbb{R}_{\geq 0 }^d$ as $n$ approaches infinity, {or if $\hat{\mathbf{r}}$ has irrational components, to approximate coefficients $[\mathbf{z}^{n \mathbf{s}_n}]F(\mathbf{z})$ for large $n$ with $\mathbf{s}_n$ tending towards $\hat{\mathbf{r}}$}.    

To analyze the Cauchy integral, {the torus $T$ can be expanded} into a cycle $\mathcal{C}$ which gets stuck on some chosen subset of the singularities of $F(\mathbf{z})$ (which are the zeroes of $H(\mathbf{z})$), and expands beyond them elsewhere.  Due to the $\mathbf{z}^{-\mathbf{r}}$ term in the integrand, we {expect} that as $\mathbf{r} \to \infty$, the integrand will decay exponentially faster in the regions of $\mathcal{C}$ away from the singularities of $F$, since the magnitude of $\mathbf{z}$ is larger in these regions.  {In this case}, we can approximate the integral by analyzing the integrand near the singularities, since the rest of the integral decays too quickly to contribute to the asymptotics.  However, {the method of expanding $T$ needs to be chosen carefully in order to ensure this works}.

{To expand $T$ successfully}, we need to minimize the maximum modulus of $\mathbf{z}^{-\mathbf{r}}$ along our contour $\mathcal{C}$.  The reason for this is as follows: we want to find a contour where the integrand attains its maximum modulus over some small interval, and then decays rapidly away from this interval.  At a point where the maximum modulus is not minimized, the argument of the $\mathbf{z}^{-\mathbf{r}}$ term will oscillate rapidly as $\mathbf{r}$ tends to infinity, which leads to cancellation near the singularity.  However, when the maximum modulus is minimized, we {can} approximate the integral in this region by using saddle point methods.

To minimize the maximum modulus, we consider the height function, $h(\mathbf{z}) := -\hat{\mathbf{r}} \cdot \mbox{Re}\, \log \mathbf{z}.$ Although this excludes the contribution from $F(\mathbf{z})$ in the integrand, $F(\mathbf{z})$ is bounded on compact sets, so $h$ still approximates the log modulus of the integrand as $\mathbf{r}$ approaches infinity in the direction of $\hat{\mathbf{r}}$.  {With the goal of expanding} the torus $T$ until it hits a singularity of $F$, we consider the values of $h$ on $\mathcal{V}:= \{\mathbf{z}:H(\mathbf{z}) = 0\}$.  On a cycle where the maximum of $h$ is minimized, the points where the maximum of $h$ is attained are saddle points of $h$.
Thus, the critical points of $h$ restricted to $\mathcal{V}$ will be candidates for the singularities that will contribute to the asymptotics.

To find the critical points of $h$, we consider a stratification of the space $\mathcal{V}$, {restricting} our attention to critical points within a certain stratum $S$.  
When $\mathcal{V}$ is a smooth manifold near a critical point, the critical point is called smooth.  In this case, the stratum $S$ is of dimension $d - 1$, and we can find $d - 1$ equations (in addition to $H = 0$) that characterize the location of the smooth critical points:
\[
r_1z_2 \frac{\partial H}{\partial z_2} = r_2 z_1 \frac{\partial H}{\partial z_1}, \ \ \ r_1z_3 \frac{\partial H}{\partial z_3} = r_3 z_1 \frac{\partial H}{\partial z_1},   \ \ \ \ldots, \ \ \  
r_1 z_d \frac{\partial H}{\partial z_d} = r_d z_1 \frac{\partial H}{\partial z_1}.
\]
When $H$ is a polynomial, the above critical point equations form a system of polynomial equations.  In this case, Gr\"{o}bner bases can help compute the critical points.  In general, it is not necessarily true that all critical points will contribute to the leading term of the Cauchy integral. {In this paper, we} require that the critical points be \emph{minimal} (described in Section \ref{Bivariate} below), which {guarantees} that they do.  {To see examples of identifying minimal critical points, see Section \ref{Examples} below.}

After determining which critical points are candidates for contributing to the asymptotics, we still must expand the torus $T$ into a cycle $\mathcal{C}$ which hugs $\mathcal{V}$ near these points.  Goresky and MacPherson show in \cite{GM1988} how Morse theory can lead to an explicit description of the domain of integration near a critical point.  In the case of generating functions without algebraic singularities, this machinery can be used to evaluate the residues of the integrals near each critical point, quickly leading to asymptotic expansions for the coefficients.  However, in the case where $H^{-\beta}$ has algebraic singularities, we rely on specific homotopies of the contour, and hence we do not need to use Morse theory to determine the domain of integration.  The portion of the contour near a particular critical point is called a \emph{quasi-local cycle}.  The asymptotics of the coefficients are thus given by a sum of integrals over these quasi-local cycles.  Below, we find the leading-term asymptotics for the coefficients.  However, these integral analyses can be used to find complete asymptotic expansions of coefficients, as in Raichev and Wilson's work in \cite{RW2008}.

\subsection{{Asymptotics Involving Algebraic Singularities}} \label{history}
In their 1990 paper \cite{FO1990}, Flajolet and Odlyzko described how to compute the asymptotics of a class of univariate generating functions with algebraic singularities.  They considered functions of the form,
\begin{equation} \label{gform}
g(z) = K(1 - z)^\alpha \left(\log(1 - z)\right)^\gamma \left(\log \log (1 - z)\right)^\delta,
\end{equation}
where $\alpha, \gamma, \delta,$ and $K$ are arbitrary real numbers, along with other related classes of functions.  Their results differed from previous results both in the class of generating functions covered, and in their method of proof.  Because we will use similar techniques in our proofs later, we take a moment to summarize their proof here.  Flajolet and Odlyzko relied on the univariate Cauchy integral formula:
\begin{displaymath}
\left[z^n\right] g(z) = \frac{1}{2 \pi i} \int_{\mathcal{C}} g(z) \frac{\D z}{z^{n + 1}}.
\end{displaymath}
Here, $\left[z^n\right] g(z)$ represents the coefficient of $z^n$ in the power series expansion of $g$, and $\mathcal{C}$ is any positively-oriented contour around the origin which does not enclose any singularities of $g(z)$.  Starting with any function $f$ such that $f(z) = O\left(|1 - z|^\alpha\right)$ as $z \to 1$, and letting $\mathcal{C}$ be a small circle around the origin, the authors expanded $\mathcal{C}$ in hopes of finding a contour which is easier to analyze.  In order to expand $\mathcal{C}$, Flajolet and Odlyzko also require the extra assumption that $f$ is analytic within the expanded contour $\mathcal{C}^*$.  As $\mathcal{C}$ expands, it must avoid not only the singularity at $1$, but also the branch cut emanating from this point.  They expand the contour {so it looks like the contour} $\mathcal{C}^*$, as shown in Figure \ref{fig:Hankel}.  {Like a Hankel contour, this contour wraps around the branch cut of $g$ and extends beyond the singularity at $1$, although $\mathcal{C}^*$ does not extend to infinity.}
\begin{figure}
\centering
\def\svgwidth{.35\linewidth}
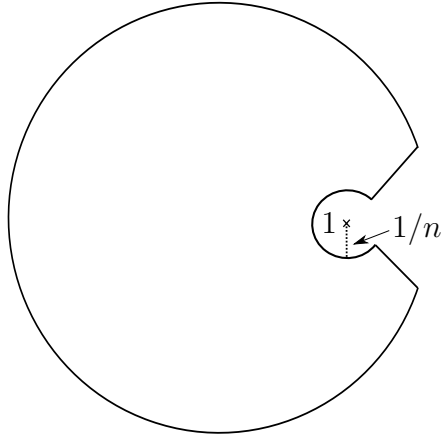
\caption{The expanded contour, $\mathcal{C}^*$, used in Flajolet and Odlyzko's proof.}
\label{fig:Hankel}
\end{figure}
From here, the contour is broken up into segments, $\gamma_1, \gamma_2, \gamma_3,$ and $\gamma_4$.  As $n$ approaches infinity, $f(z)/z^{n + 1}$, the integrand in the Cauchy integral formula, decays exponentially faster on $\gamma_4$ than it does on $\gamma_1$.  For this reason, the integral over $\gamma_4$ is negligible in the asymptotic expansion of $\left[z^n \right] f(z)$.  Likewise, the contribution along most of $\gamma_2$ and $\gamma_3$ is negligible, meaning that the asymptotics of $\left[z^n \right] f(z)$ are controlled by the integrand near $z = 1$.  However, near $z = 1$, $f(z) = O(|z - 1|^\alpha)$, which means that $f$ is bounded along the contours near the critical point, leading to the bound, $[z^n] f(z) = O(n^{-\alpha - 1})$.  Flajolet and Odlyzko then extended their results to functions $g(z)$ with the form in Equation \eqref{gform}.

{Before moving on to further developments with algebraic singularities, we highlight the connection between the proof outlined above and the results in this paper.  To analyze bivariate generating functions, we deform a torus in two complex dimensions.  Let $(p, q)$ be a point contributing to the asymptotics of such a bivariate generating function.  Then, in the proof below, one of the circles in the torus is expanded to a circle of radius $|q|$, while the other circle of the torus is expanded until it wraps around the singularity at $p$, similarly to how the Flajolet-Odlyzko contour wraps around the singularity in the univariate case.

After Flajolet and Odlyzko published their results in 1990}, other researchers extended these results to classes of multivariate generating functions.  Bender and Richmond, \cite{BR1983}, had already considered the asymptotics of multivariate generating functions with poles in 1983.  In 1992, Gao and Richmond, \cite{GR1992}, considered classes of bivariate generating functions $F(z, x)$ which are of a form they called algebraico-logarithmic, which includes some generating functions with algebraic singularities.  These algebraico-logarithmic functions could be reduced to univariate generating functions where the results of Flajolet and Odlyzko can be applied.  

Then, in his 1996 and 1998 papers, \cite{H1996} and \cite{H1998}, Hwang expanded upon the multivariate results, using a probability framework and deriving large deviation theorems.  In 1996, Hwang considered sequences of random variables $\{X_n\}$.  Assuming that the moment generating functions of the $X_n$ were of a particular form, Hwang proved a central limit theorem for $\{X_n\}$.  Then, he considered a class of bivariate generating functions $P(w, z)$ such that after approximating $\left[z^n\right]P(w, z)$ with Flajolet and Odlyzko's univariate results, $\left[z^n\right]P(w, z)$ satisfied the same conditions he required previously of the moment generating functions of $X_n$.  Applying his central limit theorem gave asymptotic results for a new class of bivariate generating functions.  In 1998, Hwang extended his results by using univariate saddle point methods to approximate integrals.

\section{{Main Result:} Bivariate Analytic Functions with Algebraic Singularities} \label{Bivariate}
In this paper, our goal is to find the asymptotics of the coefficients of $H(x, y)^{-\beta}$, where $H$ is an analytic function with real coefficients and $\beta \in \mathbb{R}$ is not a negative integer.  Let us summarize notation in a bivariate setting.  Let $\mathcal{V}$ be the zero set of the analytic function, $H(x, y)$, where $H(0, 0) \neq 0$.  We will approximate the coefficients $\left[x^ry^s\right]H(x, y)^{-\beta}$ for a fixed $\beta \in \mathbb{R}$ as $r$ and $s$ approach infinity with their ratio approaching a constant, $\lambda$.  Critical points in the direction of $\lambda = \frac{r + O(1)}{s}$ (as $r$ and $s$ approach infinity) are defined by:
%
%
%
%
%
%
\begin{align*}
H &= 0,\\
r y \frac{\partial H}{\partial y} &= s x \frac{\partial H}{\partial x}.
\end{align*}

The critical points are smooth if the gradient of $H$ does not vanish on $\mathcal{V}$ at the critical points.  Let $\mathcal{D}$ be the domain of convergence of the power series of $H^{-\beta}$ that converges around the origin, $(0, 0)$.  Then, a critical point $(p, q)$ is called minimal if $(p, q) \in \partial \mathcal{D}$.  A collection of critical points is called strictly minimal if there are no other zeroes of $H$ on $\partial \mathcal{D}$.  {For notational convenience, we will represent partial derivatives with subscripts, so that for instance, $H_x = \frac{\partial H}{\partial x}$.}  We will apply heuristics from Section \ref{MAC} to prove the following:
%
%

\begin{theorem}\label{MainResult}
Let $H$ be an analytic function with exactly $n$ strictly minimal critical points $\{(p_i, q_i)\}_{i = 1}^n$, all of which are smooth and lie on the same torus $T^*$.  (Hence, $|p_i| = |p_j|$ and $|q_i| = |q_j|$ for all $1 \leq i, j \leq n$.)   Let $\beta \in \mathbb{R}$ with $\beta \not\in \mathbb{Z}_{\leq 0}$, and let $\lambda = \frac{r + O(1)}{s}$ as $r, s \to \infty$ {with $r$ and $s$ integers}.   Define $\chi_1, \chi_2,$ and $M_i$ as follows (where $\chi_1$ and $\chi_2$ depend on $i$):
\begin{align*}
{\chi_{1, i}} &= \frac{H_y(p_i, q_i)}{H_x(p_i, q_i)} = \frac{p_i}{\lambda q_i},\\
{\chi_{2, i}} &= \frac{1}{2H_x} (\chi_1^2 H_{xx} - 2\chi_1 H_{xy} + H_{yy})\bigg|_{(x, y) = (p_i, q_i)},\\
M_i &= - \frac{2 {\chi_{2, i}}}{p_i} - \frac{{\chi_{1, i}^2}}{p_i^2} - \frac{1}{\lambda q_i^2}.
\end{align*}

For all $i$, assume $p_i, q_i, H_x (p_i, q_i),$ and $M_i$ are nonzero, and assume that the real part of {$-q_i^2M_i$} is strictly positive.  Define $\left\{x^{-\beta}\right\}_P$ as the value of $x^{-\beta}$ defined by using a ray from the origin of $\mathbb{C}$ as the branch cut of the logarithm.  In this definition, choose any ray such that $\left\{H(x, y)^{-\beta}\right\}_P = H(x, y)^{-\beta}$ in a neighborhood of the origin in $\mathbb{C}^2$ (as defined by the power series of $H^{-\beta}$), and such that this ray does not pass through $-p_i H_x(p_i, q_i)$ for any $i$.  Let $\omega_i$ be the signed number of times the curve $H(tp_i, tq_i)$ crosses this branch cut in a counterclockwise direction as $t$ increases, $0 \leq t < 1$.  Then, the following expression holds as $r, s \to \infty$:
\[
[x^ry^s] H(x, y)^{-\beta} = \sum_{i = 1}^n  \frac{r^{\beta - \frac{3}{2}} p_i^{-r} q_i^{-s} \left\{(-H_x(p_i, q_i)p_i)^{-\beta}\right\}_P e^{-\beta(2 \pi i \omega_i)}}{\Gamma(\beta) \sqrt{-2 \pi q_i^2 M_i}} + o\left(r^{\beta - \frac{3}{2}} p_1^{-r} q_1^{-s}\right).
\]
Here, the square root in the denominator is taken to be the principal root.
\end{theorem}

Unfortunately, for general $H$, the formula in Theorem \ref{MainResult} is messy, as we must find how many times the image of $H$ wraps around the origin along the path connecting $(0, 0)$ to each critical point {$(p_i, q_i)$}, and it is difficult to determine the sign of the square root.  Luckily, in the case where $H$ has only real coefficients and there is a single smooth strictly minimal critical point, we can simplify the formula.

\begin{corollary} \label{GrahamCorollary}
Let $H$ be an analytic function with a single smooth strictly minimal critical point $(p, q)$, where $p$ and $q$ are real and positive.  Let $H$ have only real coefficients in its power series expansion about the origin.  Assume $H(0, 0) > 0$, and consider $H^{-\beta}$ for $\beta \in \mathbb{R}$ with $\beta \not\in \mathbb{Z}_{\leq 0}$.  Also, define $H^{-\beta}$ here with the standard branch chosen along the negative real axis, so that $H(0, 0)^{-\beta} > 0$.  Let $\lambda = \frac{r + O(1)}{s}$ as $r, s \to \infty$ {with $r$ and $s$ integers}.  Define the following quantities:
\begin{align*}
\chi_1 &= \frac{H_y(p, q)}{H_x(p, q)} = \frac{p}{\lambda q},\\
\chi_2 &= \frac{1}{2H_x} (\chi_1^2 H_{xx} - 2\chi_1 H_{xy} + H_{yy})\bigg|_{(x, y) = (p, q)},\\
M &= - \frac{2 \chi_2}{p} - \frac{\chi_1^2}{p^2} - \frac{1}{\lambda q^2}.
\end{align*}
Assume that $H_x(p, q)$ and $M$ are nonzero.  Then, the following expression holds as $r, s \to \infty$:
\[
[x^ry^s] H(x, y)^{-\beta} \sim \frac{r^{\beta - \frac{3}{2}} p^{-r} q^{-s} (-H_x(p, q)p)^{-\beta}}{\Gamma(\beta) \sqrt{-2 \pi q^2 M}}.
\]
\end{corollary}
In the above expression, $-H_x(p, q) p$ will be a positive real number, and $(-H_x(p, q) p)^{-\beta}$ will also be a positive real number.  Additionally, $-2 \pi q^2 M$ is positive, so the positive square root is taken.

{Unfortunately, even in the simplified setting of Corollary \ref{GrahamCorollary}, it is challenging to verify that a given critical point is strictly minimal.  Section \ref{Examples} below gives a couple examples where the Corollary can be applied.  The RAGlib Maple package, \cite{S2015}, can help verify these conditions computationally.}

{Generating functions of the form in the Theorem and the Corollary are expected to appear in several contexts.  For example, there are many ways of extending the Catalan numbers to multidimensional arrays, like the Fuss-Catalan numbers.  The generating function for the Catalan numbers has a square root, and in multivariate extensions, the generating functions are still algebraic.  Another example is in counting RNA secondary structures with various structural features, called motifs.  RNA secondary structures can be analyzed using stochastic context free grammars, as in \cite{PH2014}, where multiple variables can be used to track more than one type of motif in a secondary structure simultaneously.  In such a context, the Theorem above would give asymptotics on the number of secondary structures with motifs in a fixed ratio, as the number of nucleotides in the sequence approaches infinity.
}

\section{Proof Set-Up}
To prove Theorem \ref{MainResult}, we analyze the multivariate Cauchy integral formula, Equation \eqref{MCauchy}.   {When reduced to two dimensions, the formula becomes the following:
\begin{equation} \label{BCauchy}
\left[x^ry^s\right]H(x, y)^{-\beta} = \left(\frac{1}{2 \pi i}\right)^2 \iint_{T} H(x, y)^{-\beta} x^{-r - 1} y^{-s - 1}\, \D x\, \D y.
\end{equation}
We can immediately reduce to the case where there is only one strictly minimal critical point, which we will label $(p_1, q_1) = (p, q)$.  (Correspondingly, the $i$ in the subscripts of $\chi_{1,i}, \chi_{2, i}, \omega_i$ and $M_i$ will be dropped.)  This reduction is possible because the critical points are discrete, so that the contribution from each critical point $(p_i, q_i)$ to the asymptotics is given by a quasi-local cycle disjoint from the other quasi-local cycles, and thus the contributions can be analyzed independently and then summed.  }  

{
An outline of the analysis is as follows:
\begin{enumerate}
\item In Section \ref{COV}, we find a change of variables into $(u, v)$ coordinates so that the analytic function $H(x, y)$ essentially behaves as a linear function in $u$, with some minor error terms in $v$.  This change of variables also allows us to choose a simpler quasi-local cycle near the critical point, $(p, q)$, where the $u$ and $v$ components of the contour are independent of each other.
\item In Section \ref{QLC}, we describe an appropriate expansion of the torus $T$ in the Cauchy integral formula.  Inspired by the contour from the univariate Flajolet-Odlyzko results described above, we choose a contour which similarly wraps around $p$ in the $u$-coordinate, while passing directly through $q$ in the $v$-coordinate.  The description of the contour is technical to ensure that the contour does not cross over the singularities of $H^{-\beta}$. Next, in Section \ref{away}, we verify that the region of the contour away from $(p, q)$ does not contribute to the asymptotics.
\item With the set-up complete, we are ready to analyze the Cauchy integral.  In Section \ref{productintegral}, we apply the change of variables and contour deformations from the previous step, and then justify that the integrand of the Cauchy integral is approximately the product of a function in $u$ and a function in $v$.  This step is by far the most tedious, taking many lemmas to justify, and requiring analyses along each part of the quasi-local cycle.
\item Finally, the Cauchy integral is broken up into the product of two univariate integrals, which are analyzed in Section \ref{theoremproof}.  The $u$ integral is approximately a univariate Cauchy integral, and can be related to binomial coefficients.  The $v$ integral is a standard Fourier-Laplace type integral.  Multiplying the approximations of these integrals gives the final result.
\end{enumerate}
}

\subsection{A Convenient Change of Variables} \label{COV}
In order to approximate $H(x, y)$ as a univariate linear function near the critical point $(p, q)$, it turns out it is sufficient that the power series expansion of $H$ has no constant term, linear term, nor quadratic term in one of its two input variables.  To transform $H$ into this form, we define the following change of variables:
\begin{align*}
u &= x + \chi_1(y - q) + \chi_2(y - q)^2,\\
v &= y.
\end{align*}
Here, $\chi_1$ and $\chi_2$ are as defined in Theorem \ref{MainResult} above.  Write $H$ as a power series in $u$ and $v$:
\begin{equation} \label{HPower}
H(x, y) = \sum_{m, n \geq 0} d_{mn} (u - p)^m (v - q)^n =: \tilde H(u, v).
\end{equation}
Since $H(p, q) = 0,$ we have that $d_{00} = 0$.  Notice that when $(x, y) = (p, q)$, we also have that $(u, v) = (p, q)$.  We can easily verify that $d_{01} = d_{02} = 0$ by checking some derivatives of $H$.

\subsection{Determining the Quasi-Local Cycle} \label{QLC}
For now, assume that there is a unique critical point, $(p, q)$.  Recall that the original domain of integration in Equation \eqref{BCauchy} is a torus $T$ around the origin which encloses no singularities of $H^{-\beta}(x, y)$.  To decrease the magnitude of the integrand exponentially as $r$ and $s$ approach infinity, we expand the torus $T$ towards the minimal critical point, $(p, q)$.  Because $(p, q)$ is a strictly minimal critical point, there cannot be any zeroes between the origin and $(p, q)$ that would otherwise obstruct the deformation.  Hence, we can expand the domain of integration through a homotopy until it is near the critical point.

Before expanding $T$, it is the product of a small $x$ circle and a small $y$ circle.  Begin the deformation by expanding the $y$ component to the circle, $|y| = |q|$.  The $y$ portion of the quasi-local cycle, $\mathcal{C}_y$, will be the part of this circle where $y = qe^{i \theta}$ for $|\theta| \leq \theta_y$, where $\theta_y > 0$ is a small constant.  Note that $q$ is not necessarily real.  This contour is pictured on the left in Figure \ref{fig:contour}.

\begin{figure}
\centering
\begin{subfigure}{.5\textwidth}
  \centering
  \def\svgwidth{.9\linewidth}
  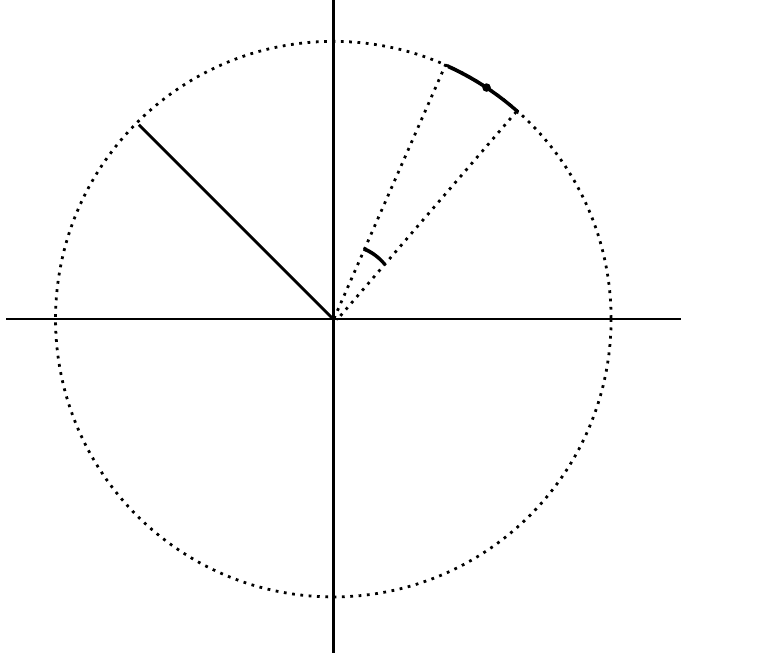
  \label{fig:sub1}
\end{subfigure}%
\begin{subfigure}{.5\textwidth}
  \centering
  \def\svgwidth{.9\linewidth}
  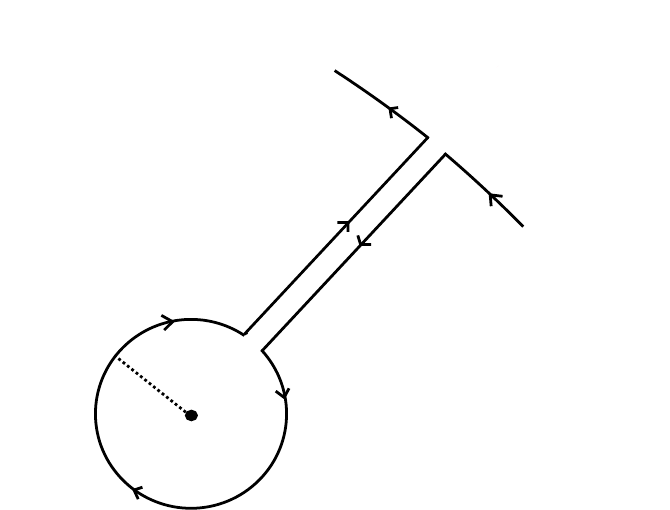
  \label{fig:sub2}
\end{subfigure}
\caption{On the left, the $y$ portion of the quasi-local contour.  On the right, a close-up of the $x$ portion of the quasi-local contour.}
\label{fig:contour}
\end{figure}

Now, for each $y \in \mathcal{C}_y$, we will expand the $x$ circle until it approaches the zero set of $H$ near $p$.  When $y$ is close to $q$, we will wrap the $x$ contour around the zero set of $H$.  However, when $y$ is further away from $q$, we will expand the $x$ contour less, so that it does not come into contact with the zero set of $H$.

More explicitly, since $H_x(p, q) \neq 0$ and $H$ is analytic, the implicit function theorem guarantees that we can parameterize the variety $\mathcal{V} = \{(x, y)| H(x, y) = 0\}$ by a smooth function $G(y)$, so that $H(p + G(y), y) = 0$ for all $y \in \mathcal{C}_y$ with $\theta_y$ sufficiently small.   So, for $y = qe^{i \theta}$ with $|\theta| \leq \frac{\theta_y}{2}$, we choose the $x$ contour appearing on the right in Figure \ref{fig:contour}.  It is not necessarily true that this contour avoids the branch cut of $H^{-\beta}$: to account for this, we can view all of our order-of-magnitude computations as if they are on the Riemann surface of $H^{-\beta}$.  Then, we can readjust our arguments accordingly when analyzing the final form of our Cauchy integral.  The equations for the pieces of the contour are as follows:
\begin{align*}
\gamma_1(y) &:= \left\{x : |x - p - G(y)| = \frac{1}{r}, \arg(p) \leq \arg(x - p - G(y)) \leq \arg(p) + 2\pi\right\},\\
\gamma_2(y) &:= \left\{x : \frac{1}{r} \leq |x - p - G(y)| \leq \epsilon_x, \arg(x - p - G(y)) = \arg(p) + 2\pi\right\},\\
\gamma_3(y) &:= \left\{x : \frac{1}{r} \leq |x - p - G(y)| \leq \epsilon_x, \arg(x - p - G(y)) = \arg(p) \right\},\\
\gamma_4(y) &:= \left\{x : |x - G(y)| = |p| + \epsilon_x, \arg(p) - \theta_x \leq \arg(x - G(y)) \leq \arg(p)\right\},\\
\gamma_5(y) &:= \left\{x : |x - G(y)| = |p| + \epsilon_x, \arg(p) \leq \arg(x - G(y)) \leq \arg(p) + \theta_x \right\}.
\end{align*}
{Here, $\epsilon_x > 0$ is a small positive constant.}  Note that later on, the change of variables into $(u, v)$ coordinates will allow us to drop the corresponding $G(y)$ term in the contour when sufficiently close to the critical point $(p, q)$.  Now, as $|\theta|$ increases with $|\theta| \geq \frac{\theta_y}{2}$, we would like to find an interpolation of the $x$ quasi-local contour, shrinking it until it no longer wraps around the zero set of $H$. 
To do this, notice that when $y = qe^{i\theta_y t}$ for $t \in \left[-1, -\frac{1}{2}\right] \cup \left[ \frac{1}{2}, 1\right]$, $|p + G(y)| > |p|$ uniformly, since $(p, q)$ is a strictly minimal critical point of $H$.  Therefore, we can find a $\delta > 0$ so that $|p + G(y)| > |p| + \delta$ for every $t \in \left[-1, -\frac{1}{2}\right] \cup \left[ \frac{1}{2}, 1\right]$.  

{For $y = qe^{i \theta}$ with $|\theta| > \frac{\theta_y}{2}$, we linearly interpolate the radius $|x - G(y)|$ in $\gamma_4$ and $\gamma_5$ from $|p| + \epsilon_x$ to $|p| + \delta$ as $|\theta|$ increases from $\frac{\theta_y}{2}$ to $\theta_y$, while correspondingly adjusting each other part of the contour, $\gamma_i$, to form a closed curve as necessary.  This gradually shrinks the quasi-local contour until it no longer wraps around the zero set $\mathcal{V}$.}
%
 We will show that the integrand is small along all parts of this contour, so the details of {how the $\gamma_i$ intersect} are not important.

This completes the description of a possible quasi-local contour near $(p, q)$, but we will morph it slightly so that it is more convenient.  Consider applying the change of variables given in Section \ref{COV}.  Since $v = y$, the $v$ portion of the contour is identical to the $y$ portion of the contour.  Then, since $u = x + \chi_1(v - q) + \chi_2(v - q)^2$, each contour $\gamma_i(y)$ is translated by $\chi_1(v - q) + \chi_2(v - q)^2$, so that it retains its overall shape but is centered at a new location.  Additionally, the parameterization of the zero set $G(y)$ changes to a new parameterization $\kappa(v)$ (discussed below) such that $\tilde H(p + \kappa(v), v) = 0$.  $\kappa(v)$ oscillates slower than $G(y)$ near the critical point.  {(Here, $\tilde H(u, v) = H(x, y)$, as defined in Equation \eqref{HPower}.)}  This will allow us to approximate our quasi-local cycle as a product contour near the critical point $(p, q)$, because we will show that we can drop $\kappa(v)$ from the contour for $y$ close enough to $q$.

In summary, the final quasi-local cycle $\mathcal{C}(p, q)$ (in $(u, v)$-coordinates) near the critical point $(p, q)$ has three regimes.  The contour is an arc in $v$, and wraps around the zero set of $\tilde H$ in $u$.  Let $v = qe^{i \theta}$.  When $\theta \leq r^{-\frac{2}{5}}$, the $u$-contour wraps exactly around the point, $p$, and this portion of the contour is a product contour.  When $r^{-\frac{2}{5}} \leq \theta \leq \frac{\theta_y}{2}$, the contour instead wraps around the point $p + \tilde \kappa(v)$ for a suitably chosen interpolation $\tilde \kappa$.  Finally, if $\theta \geq \frac{\theta_y}{2}$, then the $u$-contour gradually shrinks as $\theta$ increases, until it no longer intersects the zero set of $\tilde H$ at all.

Here, we discuss the parameterization, $\kappa(v)$, and interpolation, $\tilde \kappa(v)$, in more detail.  Using the chain rule, {we have $\tilde H_u(p, q) = H_x(p, q)$}. Since $H_x(p ,q) \neq 0$ by assumption, and since $\tilde H$ is analytic, the implicit function theorem guarantees that there exists a smooth parameterization $\kappa(v)$ of the zero set of $\tilde H$, so that $\tilde H(p + \kappa(v), v) = 0$ for $v$ sufficiently close to $q$.  Investigating $\kappa(v)$ a little further, we use the power series expansion of $\tilde H$ about $(p, q)$, given in Equation \eqref{HPower}, along with the facts that $d_{00} = d_{01} = d_{02} = 0$, to obtain the following:
\[
0 = \tilde H(p + \kappa(v), v) = d_{10} \kappa(v) + O(\kappa(v))^2 + O(v - q)^3.
\]
Thus, $\kappa(v) = O(v - q)^3$.  {However, by comparing $\kappa(v)$ to $G(y)$, we find that $G(y) = -\chi_1(y - q) - \chi_2(y - q)^2 + O(y - q)^3$.}
  This provides more insight into this change of variables: in addition to allowing us to write $H$ as a nice power series with some vanishing coefficients, the change of variables also describes $\mathcal{V}$ near $(p, q)$.  By converting the contour into $(u, v)$-coordinates, we are able to stabilize the $u$ contours, slowing down the movement of the zero set of $H$ when it is parameterized by $v$.  We take advantage of this slow-down by morphing our contour slightly, as described in the following paragraph.

  In order to break the $2$-dimensional Cauchy integral into two one-dimensional integrals, we need the quasi-local contour to be a product contour near the critical point, $(p, q)$.  To achieve this goal, we will need to break into two cases: when $|\theta| \leq r^{-\frac{2}{5}}$ and when $|\theta| > r^{-\frac{2}{5}}$, for $v = qe^{i \theta}$.  Let us first analyze $|v - q|$ in these cases:
\begin{align}
v - q &= qe^{i \theta} - q \nonumber \\
&= qi\theta - \frac{q\theta^2}{2} + O(\theta)^3. \label{v-q}
\end{align}  
  In the second line, we use the power series expansion for $e^{i \theta}$, which holds uniformly as $\theta \to 0$.
  Now, if $|\theta| \leq r^{-\frac{2}{5}}$, then $|(v - q)^3| = O\left(r^{-\frac{6}{5}}\right) < \frac{1}{r}$ for $r$ sufficiently large.  Hence, when $|\theta| \leq r^{-\frac{2}{5}}$, $|\kappa(v)| = O\left(r^{-\frac{6}{5}}\right)$.  Therefore, for $r$ sufficiently large, the point $p + \kappa(v)$ is always within the circle of radius $\frac{1}{r}$ about the point $p$, and we can morph our $u$-contour so that it is centered exactly around the point $p$ instead of the point $p + \kappa(v)$.  Thus, we will drop $\kappa(v)$ from the definitions of all the $\gamma_i$ when $\theta \leq r^{-\frac{2}{5}}$, which means that the $u$ contour no longer depends on $v$ when $\theta \leq r^{-\frac{2}{5}}$.  (Note that this corresponds to a similar shift in the original $(x, y)$-coordinates, which can be computed explicitly to justify that the original torus $T$ can be morphed locally to this new contour.)  The portion of the contour where $\theta \leq r^{-\frac{2}{5}}$ will yield the dominating contribution to the integral asymptotically.

In the other regime, when $\theta \geq r^{-\frac{2}{5}}$, we cannot simply eliminate $\kappa(v)$.  Instead, let $\tilde \kappa(v)$ be $0$ when $\theta \leq r^{-\frac{2}{5}}$, let it be $\kappa(v)$ when $\theta \geq r^{-\frac{7}{20}}$, and let it linearly interpolate between $0$ and $\kappa(v)$ when $r^{-\frac{2}{5}} \leq \theta \leq r^{-\frac{7}{20}}$.  We replace $\kappa(v)$ with $\tilde \kappa(v)$ in the definition of the quasi-local cycle.  Note that $\tilde \kappa(v) = O(v - q)^3$ as $v$ tends to $q$.  We will use this condition much later in the proof.  This completes the description of the quasi-local cycle.

\subsection{Away from the Quasi-Local Cycle} \label{away}

Let us justify that the integral over the quasi-local cycle provides the main contribution to the asymptotics of the coefficients, and that the remainder of the domain of integration contributes negligibly in comparison.  In order to do so, we will find a way to expand the torus $T$ away from the quasi-local cycle so that the integrand decays exponentially faster here when compared to the quasi-local cycle.  To begin, consider the case where there is only one strictly minimal critical point, $(p, q)$.  Formally, by strictly minimal, we mean the following:
\[
\{|x| \leq |p|\} \cap \{|y| \leq |q|\} \cap \mathcal{V} = (p, q).
\]
Here again, $\mathcal{V} = \{(x, y) | H(x, y) = 0\}$.

Consider the torus, $T_{(p, q)} := \{x: |x| = |p|\} \times \{y: |y| = |q|\}$.  From this torus, remove an open neighborhood $\mathcal{N}$ of the point $(p, q)$, where $\mathcal{N}$ is so small that the angular sectors of the torus that it covers in $x$ and $y$ are smaller than the angular sectors of the torus that $\mathcal{C}$ covers in $x$ and $y$.  That is, the $y$ component of $\mathcal{N}$ should only consist of $y$ values whose arguments $|\arg(y) - q| < c < \theta_y$ for some constant $c > 0$.  Similarly, for each $y \in \mathcal{C}$, the arguments of the $x$ values in $\mathcal{N}$ should not vary from $p + G(y)$ more than the arguments of the $x$ values in $\mathcal{C}$.

Now, $T_{(p, q)} \backslash \mathcal{N}$ is a closed set which does not intersect the closed set, $\mathcal{V}$.  Thus, there are open sets dividing these two sets.  This implies that there is a neighborhood of $T_{(p, q)} \backslash \mathcal{N}$ which does not intersect $\mathcal{V}$.  There is some $\delta^* > 0$ such that the $x$ arc of $T_{(p, q)}\backslash \mathcal{N}$ can be expanded by $\delta^*$ without hitting $\mathcal{V}$.

Then, at every point of this new cycle away from the critical point $(p, q)$, we have that $|x| \geq |p| + \delta^*$.  This forces the Cauchy integral to decay exponentially faster away from the critical point than it does near $(p, q)$, proving that the asymptotic contribution to the integral cannot come from $T_{(p + \delta^*, q)}$.

After expanding $T_{(p, q)}$ to $T_{(p + \delta^*, q)} :=  \{x: |x| = |p| + \delta^*\} \times \{y: |y| = |q|\}$, notice that $T_{(p + \delta^*, q)} \backslash \mathcal{N}$ can be connected to the quasi-local cycle $\mathcal{C}$ by adding two short lines at the ends of the $x$ contours connecting the circle of radius $\delta^*$ to the ends of $\gamma_4$ and $\gamma_5$.  Because these lines are contained entirely within the region near $(p, q)$ where the implicit function theorem holds for $G(y)$, the lines cannot hit any zeroes of $H$.  Also, the magnitude of $\left|xy^{\frac{1}{\lambda}}\right|$ along these lines is always greater than the magnitude of $\left|pq^{\frac{1}{\lambda}}\right|$ because $|x| > |p|$, which means that these lines also do not contribute to the asymptotics of the integral, and may be ignored.

\section{Approximating with a Product Integral} \label{productintegral}
With our quasi-local cycle and change of variables defined, we are ready to begin analyzing the Cauchy integral formula, Equation \eqref{BCauchy}.  Writing the quasi-local cycle near $(p, q)$ as $\mathcal{C}(p, q)$, we apply the change of variables to the Cauchy integral formula restricted to $\mathcal{C}(p, q)$ to obtain the following integral:
\begin{equation} \label{ToBeProduct}
\left(\frac{1}{2\pi i}\right)^2 \iint_{\mathcal{C}(p, q)} \tilde H(u, v)^{-\beta} \left(u - \chi_1(v-q) - \chi_2(v - q)^2\right)^{-r - 1} v^{-s - 1} \, \D u \, \D v.
\end{equation}
Here, we used the fact that the Jacobian of the transformation is $1$.  Our goal now is to show that this integral is essentially a product integral.  The following lemma describes this precisely.
%
%
\begin{lemma} \label{product} {The integral in Equation \eqref{ToBeProduct} is asymptotically equivalent to the following:}
\[ \left( \frac{1}{2 \pi i}\right)^2 \iint_{\mathcal{C}_\ell(p, q)} [H_x(p, q) \cdot (u - p)]^{-\beta} u^{-r - 1} v^{-s - 1} \left[1 - \frac{\chi_1 (v - q) + \chi_2 (v - q)^2}{p} \right]^{-r - 1} \,\D u \,\D v.
\]
The above estimate holds as $r, s \to \infty$ with $\lambda = \frac{r + O(1)}{s}$.  Here, $\mathcal{C}_{\ell}(p, q)$ is the portion of $\mathcal{C}(p, q)$ where $|\theta| \leq r^{-\frac{2}{5}}$.  Hence, $\mathcal{C}_{\ell}(p, q)$ is a product contour.
\end{lemma}

The proof of this lemma involves two types of statements: near the critical point, where $|u - p|$ and $|v - q|$ are both sufficiently small, we will argue that the integrands are asymptotically the same.  Away from the critical point, where at least one of $|u - p|$ or $|v - q|$ is sufficiently large, we will show that both integrands are small, and hence do not contribute asymptotically to either integral.  (In the second integral, we need only show that the integrand is small when $|u - p|$ is large, since $|v - q|$ is always small in $\mathcal{C}_\ell(p, q)$.)

\subsection{{When} $|u - p|$ and $|v - q|$ are Small}

In order to match the two integrands when $|v - q|$ and $|u - p|$ are small, we rewrite the original:
\begin{multline*}
\tilde H(u, v)^{-\beta} \left(u - \chi_1(v - q) - \chi_2 (v - q)^2\right)^{-r - 1} v^{-s - 1}
\\
 = [H_x(p, q) \cdot (u - p)]^{-\beta} u^{-r - 1} v^{-s - 1} \left[1 - \frac{\chi_1 (v - q) + \chi_2 (v - q)^2}{p} \right]^{-r - 1}  K(u, v) L(u, v).
\end{multline*}
Here, $K$ and $L$ are correction factors with the following definitions:
\[
K(u, v) := \left( \frac{1 - \frac{\chi_1 (v - q) + \chi_2 (v - q)^2}{u}}{1 - \frac{\chi_1 (v - q) + \chi_2 (v - q)^2}{p}}\right)^{r - 1},  \ \ \ \ \ 
L(u, v) := \left[ \frac{\tilde H(u, v)}{H_x(p, q)(u - p)}\right]^{-\beta}.
\]
Thus, our goal is to show that $K$ and $L$ are asymptotically {equivalent} to $1$.  We can analyze $K$ and $L$ along each part of the contour where $(u - p)$ and $(v - q)$ are small, and verify that in every case, they are asymptotically equal to $1 + o(1)$ as $r, s \to \infty$ with $\lambda = \frac{r + O(1)}{s}$.  Explicitly, we will show this for $u$ in $\gamma_1$, and for the parts of $\gamma_2$ and $\gamma_3$ sufficiently close to the critical point.  {Recall that $\mathcal{C}_y$ is the $v$-portion of the quasi-local contour, shown on the left in Figure \ref{fig:contour}, where $v = qe^{i \theta}$ for $|\theta| \leq \theta_y$ with $\theta_y > 0$ a small constant.}

\begin{lemma} \label{Kg1}
Assume $v \in \mathcal{C}_y$ with $|\theta| \leq r^{-\frac{2}{5}}$.  Also, assume that either $u \in \gamma_1$, or that $u \in \gamma_2 \cup \gamma_3$ with $u = p + \frac{\omega t}{r}$ and $t \leq r^{\frac{3}{10}}$.  Then, the following holds uniformly as $r, s \to \infty$ with $\lambda = \frac{r + O(1)}{s}$:
\[
K(u, v) = 1 + o(1).
\]
\end{lemma}

\begin{proof}
We pull aside the numerator of $K(u, v)$:
\[
1 - \frac{\chi_1(v - q) + \chi_2(v - q)^2}{u} = 1 - \frac{\chi_1(v - q) + \chi_2(v - q)^2}{p} \cdot \frac{1}{1 - \left(1 - \frac{u}{p}\right)}.
\]
For $u \in \gamma_1$, $|u - p| = \frac{1}{r}$.  Thus, we have $\left(1 - \frac{u}{p}\right) = O\left(r^{-1}\right)$, and $\left|1 - \frac{u}{p}\right| < 1$ for $r$ sufficiently large.  Hence, we can expand $\frac{1}{1 - \left(1 - \frac{u}{p}\right)}$ as a uniformly convergent geometric series for all $u \in \gamma_1$.  This yields the following:
%
%
\begin{multline} \label{KNum}
1 - \frac{\chi_1(v - q) + \chi_2(v - q)^2}{u} \\
= 1 - \frac{\chi_1 (v - q) + \chi_2(v - q)^2}{p}\left[1 + \left(1 - \frac{u}{p}\right) + \left(1 - \frac{u}{p}\right)^2 + \cdots \right].
\end{multline}
Now, we can replace the numerator in the base of $K$ by the expression in Equation \eqref{KNum} to obtain the following:
\begin{equation} \label{KTaylor}
\frac{1 - \frac{\chi_1(v - q) + \chi_2(v - q)^2}{u}}{1 - \frac{\chi_1(v - q) + \chi_2(v - q)^2}{p}} = 1 - \frac{\frac{\chi_1(v - q) +\chi_2(v - q)^2}{p}}{1 - \frac{\chi_1(v - q) + \chi_2(v - q)^2}{p}} \left[\left(1 - \frac{u}{p}\right) + O\left(1 - \frac{u}{p}\right)^2\right].
\end{equation}
%
%
Equation \eqref{KTaylor} holds uniformly for $|\theta| \leq r^{-\frac{2}{5}}$ and $u$ in the region of the lemma, as $r \to \infty$.  Between $\gamma_1$ and the regions of $\gamma_2$ and $\gamma_3$ described in the lemma, $\left(1 - \frac{u}{p}\right) = O\left(r^{-\frac{7}{10}}\right)$.  Also, from Equation \eqref{v-q}, $|v - q| = O\left(r^{-\frac{2}{5}}\right)$.  Plugging these facts into Equation \eqref{KTaylor} yields the following:
\[
\frac{1 - \frac{\chi_1(v - q) + \chi_2(v - q)^2}{u}}{1 - \frac{\chi_1(v - q) + \chi_2(v - q)^2}{p}} = 1 + O\left(r^{-\frac{11}{10}}\right).
\] 
We replace the base of $K$ with this new expression {and use the Taylor series for the natural logarithm} to obtain:
\[
K(u, v) = \left(1 + O\left(r^{-\frac{11}{10}}\right)\right)^{-r - 1} = e^{(-r - 1) \ln \left(1 + O\left(r^{-\frac{11}{10}}\right)\right)} = e^{(-r - 1) \cdot O\left(r^{-\frac{11}{10}}\right)} = 1 + o(1).
\]
\end{proof}

Next, we prove the corresponding statement for $L(u, v)$ on $\gamma_1$ and the parts of $\gamma_2$ and $\gamma_3$ sufficiently close to $p$.

\begin{lemma}
Assume $v \in \mathcal{C}_y$ with $|\theta| \leq r^{-\frac{2}{5}}$.  Also, assume either that $u \in \gamma_1$, or that $u \in \gamma_2 \cup \gamma_3$ with $u = p + \frac{\omega t}{r}$ and $t \leq r^{\frac{3}{10}}$.  Then, the following holds uniformly as $r, s \to \infty$ with $\lambda = \frac{r + O(1)}{s}$:
\[
L(u, v) = 1 + o(1).
\]
\end{lemma}

\begin{proof}
Recall that $\tilde H(u, v)$ has a particularly nice power series, given in Equation \eqref{HPower}:
\[
\tilde H(u, v) = \sum_{m, n \geq 0} d_{mn} (u - p)^m (v - q)^n.
\]
In this series, we have the restrictions, $d_{00} = d_{01} = d_{02} = 0$.  Hence, we can express $\tilde H$ in the following manner:
\begin{equation} \label{HApprox}
\tilde H = d_{10} (u - p) + f(u, v) + g(u, v) + h(u, v).
\end{equation}
{Here, we choose $f, g,$ and $h$ to be any functions satisfying Equation \eqref{HApprox} such that $f(u, v) = O(u - p)^2, g(u, v) = O\left((u - p)(v - q)\right),$ and $h(u, v) = O(v - q)^3$, each uniformly as $(u, v)$ approaches $(p, q)$.  Also, we recall that in the power series expansion of $\tilde H$, $d_{10} = H_x(p, q)$.}

We now plug Equation \eqref{HApprox} into the definition of $L$:
\begin{align}
L(u, v) &:= \left[ \frac{\tilde H(u, v)}{H_x(p, q) (u - p)}\right]^{-\beta} \nonumber \\
&= \left[ 1 + \frac{f + g + h}{H_x(p, q)(u - p)}\right]^{-\beta}. \label{LApprox}
\end{align}
In the region described in this Lemma, we have the restrictions, $\frac{1}{r} \leq |u - p| \leq r^{-\frac{7}{10}}$, and $|v - q| = O\left(r^{-\frac{2}{5}}\right)$.  Thus, we obtain the following expressions:
\begin{align*}
\frac{{f(u, v)}}{H_x(p, q) (u - p)} &= O(u - p) = O\left(r^{-\frac{7}{10}}\right),\\
\frac{{g(u, v)}}{H_x(p, q)(u - p)} &= O(v - q) = O\left(r^{-\frac{2}{5}}\right),\\
\frac{{h(u, v)}}{H_x(p, q)(u - p)} &= O\left( \frac{(v - q)^3}{u - p}\right) = O\left(r \cdot \left(r^{-\frac{2}{5}}\right)^3\right) = O\left(r^{-\frac{1}{5}}\right).
\end{align*}
Each of these statements holds uniformly over the region in the lemma as $r \to \infty$.  Plugging these into Equation \eqref{LApprox} above yields the desired result:
\[
L(u, v) = \left[1 + O\left(r^{-\frac{1}{5}}\right)\right]^{-\beta} = 1 + o(1).
\]
\end{proof}

This completes the proof that our integrand is essentially a product integrand near the critical point.  It remains to show that the contributions away from the critical point are negligible.

\subsection{{When} $(u - p)$ or $(v - q)$ is Big}
In order to finish justifying that the {integral in Equation \eqref{ToBeProduct}} is approximately a product integral, we must show both that the {integral in Equation \eqref{ToBeProduct}} and the product integral {in Lemma \ref{product}} are small away from the critical point, $(p, q)$.  We show this by approximating the magnitude of the integrands on the portion of the contours away from $(p, q)$, and seeing that the contributions to the integral away from the critical point are negligible in comparison to the final approximation for the coefficients of $H^{-\beta}$.  We look at the integrals separately.
%
%
%
%
%

\begin{lemma} \label{vbig1}
Let $\mathcal{\bar C}(p, q)$ represent the portion of $\mathcal{C}(p, q)$ where at least one of the following conditions holds: $|\theta| > r^{-\frac{2}{5}}$ or $|u - q| \geq r^{-\frac{7}{10}}$.  Then, the following holds uniformly as $r, s \to \infty$ with $\lambda = \frac{r + O(1)}{s}$ {for some constant $d > 0$}:
\begin{align*}
\left(\frac{1}{2 \pi i}\right)^2 \iint_{\mathcal{\bar C}(p, q)} \tilde H(u, v)^{-\beta} \left(u - \chi_1(v - q) - \chi_2(v - q)^2\right)^{-r - 1}&v^{-s - 1} \, \D u \, \D v \\
&= O\left(p^{-r}q^{-s} r^{|\beta|} e^{-\frac{d}{2} r^{\frac{1}{5}}}\right).
\end{align*}
\end{lemma}

\begin{proof}
We bound the terms of the integrand separately.  First, recall the nice power series, $\tilde H(u, v) = \sum_{m, n \geq 0} d_{mn} (u - p)^m(v - q)^n$, with the relations, $d_{00} = d_{01} = d_{02} = 0$ and $\tilde H(p + \kappa(v), v) = 0$.  {Recall that $\kappa(v)$ is the parameterization of the zero set of $\tilde H(u, v)$ near $(p, q)$, and that $\tilde \kappa(v)$ is the interpolation of this parameterization that is zero sufficiently close to $(p, q)$ and $\kappa(v)$ further from $(p, q)$, as described in Section \ref{QLC}.}  Define $\bar u$ by $\bar u = u - p - \tilde \kappa(v)$ and $\bar v$ by $\bar v = v - q$.  $\tilde H(p + \kappa(v), v)$ can be represented as follows:
\[
0 = \tilde H(p + \kappa(v), v) = d_{10} \kappa(v) + d_{11} \kappa(v) \bar v + d_{20} \kappa(v)^2 + d_{03} \bar v^3 + \cdots.
\]
With this in mind, we {plug the point $(p + \tilde \kappa(v) + \bar u, v)$ into the power series of $\tilde H$ and extract the portion corresponding to $\tilde H(p + \kappa(v), v)$:}
\begin{align*}
\tilde H(p + \tilde \kappa(v) + \bar u, v) =&\   \tilde H(p + \kappa(v), v) + d_{10}\big([\tilde \kappa - \kappa](v) + \bar u\big)\\
& + O\big([\tilde \kappa - \kappa](v)\big)^2 + O(\bar u)^2 + O\big([\tilde \kappa - \kappa](v) \bar u\big)\\
& + O\big(\kappa(v) [\tilde \kappa - \kappa](v)\big) + O\big(\kappa(v) \bar u\big)\\
& + O\big([\tilde \kappa - \kappa](v) \bar v\big) + O\big(\bar u \bar v\big).
\end{align*}

Recall that $[\tilde \kappa - \kappa](v) = O\left(\bar v^3\right)$ and $\kappa(v) = O\left(\bar v^3\right)$.  For $|\theta| \leq r^{-\frac{7}{20}}, \bar v = O\left( r^{-\frac{7}{20}}\right)$ by Equation \eqref{v-q}, so that $[\tilde \kappa - \kappa](v) = O\left(r^{-\frac{21}{20}}\right)$.  However, for $|\theta| \geq r^{-\frac{7}{20}},$ $\tilde \kappa$ is exactly $\kappa$.  Thus, $[\tilde \kappa - \kappa](v) = O\left(r^{-\frac{21}{20}}\right)$ for all $|\theta| \leq \theta_y$.  Additionally, $|\bar u| \geq \frac{1}{r}$ on all parts of $\mathcal{\bar C}$.  Therefore, for $\epsilon_x$, $\theta_y$, and $\theta_x$ sufficiently small, all terms in the expansion of $\tilde H$ are negligible except $d_{10} \bar u$.  Since $|\bar u| \geq \frac{1}{r}$ and {since $H$ is bounded on compact sets}, we have the following bound uniformly on $\mathcal{\bar C}$:
\begin{equation} \label{HBound}
\tilde H^{-\beta} = O\left(r^{|\beta|}\right).
\end{equation}

Now, we turn to the remaining part of the integrand.  Using the relation, $s = \frac{r}{\lambda} + O(1)$ as $r, s \to \infty$, we have the following:
\begin{multline} 
\left(u - \chi_1(v - q) - \chi_2(v - q)^2\right)^{-r-1} v^{-s-1} \\
= p^{-r} q^{-\frac{r}{\lambda}} \left(u - \chi_1(v - q) - \chi_2(v - q)^2\right)^{-1} v^{-1 + O(1)} e^{-r \varphi(u, v)}.
\label{integrand}
\end{multline}
Here, we also took a factor of $p^{-r}$ out of $\left(u - \chi_q(v - q) - \chi_2(v - q)^2\right)^{-r}$ and a factor of $q^{-\frac{r}{\lambda}}$ out of $v^{-\frac{r}{\lambda}}$.  $\varphi$ is defined by $\varphi(u, v) = \ln \left( \frac{1}{p}\left[u - \chi_1(v - q) - \chi_2(v - q)^2\right]\right) + \lambda^{-1} \ln\left[\frac{v}{q}\right]$.  We can expand $\varphi$ as a bivariate power series:
\begin{equation} \label{phidef}
\varphi(u, v) = \frac{1}{p}(u - p) + \frac{M}{2}(v - q)^2 + O\big((u - p)(v - q)\big) + O(u - p)^2 + O(v - q)^3.
\end{equation}
This equation holds uniformly as $(u, v)$ approaches $(p, q)$. {$M$ is a constant in terms of the derivatives of $H$, as defined in the statement of Theorem \ref{MainResult}.}  Rewriting $\varphi(u, v)$ in terms of $\tilde \kappa, \bar u,$ and $\bar v = qi\theta + O(\theta)^2$ gives the following as $\bar u, \bar v \to 0$:
\begin{align*}
\varphi(u, v) &= \varphi(p + \tilde \kappa(v) + \bar u, q + c \bar v)\\
 &= \frac{1}{p}(\tilde \kappa(v) + \bar u) + \frac{M}{2}\left(\bar v\right)^2 + O\big((\tilde \kappa (v) + \bar u)(\bar v)\big) + O(\tilde \kappa (v) + \bar u)^2 + O(\bar v)^3\\
&=  \frac{1}{p} \bar u + \frac{M}{2}\left( \bar v\right)^2 + O(\bar u \bar v) + O(\bar u)^2 + O(\bar v)^3\\
&= \frac{1}{p} \bar u - \frac{q^2M}{2} \theta^2 + O(\theta \bar u) + O(\bar u)^2 + O(\theta)^3.
\end{align*}

  From here, our goal is to bound $e^{-r \varphi}$ in magnitude.  To do so, we will investigate the real part of $\varphi$.  Let $d = \re\left( -\frac{q^2 M}{2}\right)$, which is a strictly positive number by assumption.  We break into cases now.\\

\begin{itemize}

\hrule
\item[{\bf Case 1:}]
Consider the case where $u$ is close to the critical point $p$, in the sense that either $u \in \gamma_1$ and $|\theta| \geq r^{-\frac{2}{5}}$, {or $u \in \gamma_2$ or $\gamma_3$ and $|\bar u| \leq r^{-\frac{7}{10}}$ but $|\theta| \geq r^{-\frac{2}{5}}$.  Either $|\bar u| = \frac{1}{r}$ for $u \in \gamma_1$, which is much smaller than $\theta^2$, or $\frac{1}{p} \bar u$ is a strictly positive real number if $u \in \gamma_2 \cup \gamma_3$.  In either case, $u$ at worst does not increase the real part of $\varphi$, and we obtain the following:}
\[
\re \left( \varphi(u, v) \right) \geq \re\left( -\frac{q^2M}{2}\theta^2 + o(\theta)^2 \right) \geq \frac{dr^{-\frac{4}{5}}}{2}.
\]
The above inequality holds for $r$ sufficiently large and for $\epsilon_x$ and $\theta_y$ small enough.  Thus, we have for $r$ sufficiently large:
\[
\left|e^{-r \varphi(u, v)}\right| \leq e^{-\frac{d}{2}r^{\frac{1}{5}}}.
\]

\hrule
\item[{\bf Case 2:}] Consider the case where $u \in \gamma_2$ or $\gamma_3$ and $|\bar u| \geq r^{-\frac{3}{10}}$.  (This case is only relevant when $|\theta|$ is small enough for $\gamma_2$ and $\gamma_3$ to be part of the contour.)  For sufficiently small $\epsilon_x$ and $\theta_y$,  the $O(\bar u \theta)$ term is dominated by the $\bar u$ term.  The remaining $\theta$ terms are dominated by the $\theta^2$ term, so these $\theta$ terms can only increase the real part of $\varphi$.  Thus, the real part of $\varphi$ is at least half the $\frac{1}{p} \bar u$ term, and we have the following for $r$ sufficiently large:
\[
\re \left(\varphi(u, v)\right) \geq \frac{1}{2|p|} r^{-\frac{3}{10}}.
\]
Plugging this into the exponential yields the following:
\[
\left|e^{-r \varphi(u, v)}\right| \leq e^{-\frac{1}{2|p|} r^{\frac{7}{10}}}.
\]

\hrule
\item[{\bf Case 3:}] {Now, consider the case where $u \in \gamma_4$ or $\gamma_5$.  If $|\theta| \leq \frac{\theta_y}{2}$, then $|u - \tilde \kappa(v)| = |p| + \epsilon_x$.  Similarly, if $|\theta| \geq \frac{\theta_y}{2}$, then $|u_1 - \tilde \kappa(v)| \geq |p_1| + \min\{\delta, \epsilon_x\}$.  Let $\mathcal{E} = \min\{\delta, \epsilon_x\}$.  Then,
\begin{equation} \label{umag}
\left| \frac{1}{p} \bar u \right| \geq \left| \frac{1}{p} \right| [|u - \tilde \kappa(v)| - |p|] \geq \frac{\mathcal{E}}{|p|}.
\end{equation}
Also, for $\theta_x$ sufficiently small (depending on $\epsilon_x$ and $|p|$), the following holds:
\begin{equation} \label{uarg}
\left| \arg(u - p - \tilde\kappa(v)) - \arg(p)\right| \leq \frac{\pi}{3}.
\end{equation}
This statement should be clear graphically: let $\alpha = \arg\left(\frac{u}{p} - 1 - \frac{\tilde\kappa(v)}{p}\right)$, and consider Figure \ref{fig:alpha}.
\begin{figure}
\centering
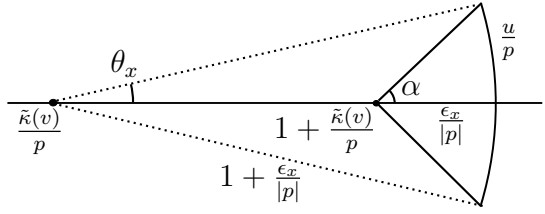
\caption{$\alpha$ must be small when $\theta_x$ is small.}
\label{fig:alpha}
\end{figure}
Clearly, as $\theta_x$ tends to zero, $\alpha$ approaches zero as well, verifying Equation \eqref{uarg}.  Combining Equation \eqref{umag} and Equation \eqref{uarg}, we have the following:
\[
\re \left[\frac{1}{p} \bar u\right] \geq \frac{\mathcal{E}}{|p|} \cos\left(\arg(u - p - \tilde \kappa(v)) - \arg(p)\right) \geq \frac{\mathcal{E}}{2|p|}.
\]
Just like in Case 2,} in the expansion of $\varphi$, the $O(\bar u \theta)$ term is dominated by the $\bar u$ term, and the remaining $\theta$ terms are dominated by the $\theta^2$ term, which only adds to the real part of $\varphi$.  Hence, for $\epsilon_x, \theta_x,$ and $\theta_y$ sufficiently small we have:
\[
\re (\varphi(u, v)) \geq \frac{\mathcal{E}}{4|p|}.
\]
This yields:
\[
\left|e^{-r \varphi(u, v)}\right| \leq e^{-\frac{\epsilon_x}{4|p|} r}.
\]
This decay is much greater than in the other cases: this is because here, $(u, v)$ is bounded away from $(p, q)$ by a constant amount.
\end{itemize}
\hrule
In every case, we have the following bound for $\epsilon_x, \theta_x, \theta_y,$ and $\delta$ sufficiently small and $r$ sufficiently large:
\begin{equation} \label{PhiBound}
\left|e^{-r \varphi(u, v)}\right| \leq e^{-\frac{d}{2} r^{\frac{1}{5}}}.
\end{equation}
Finally, notice that for $\epsilon_x$, $\theta_y$, $\theta_x$, and $\delta$ sufficiently small,
\begin{equation} \label{EasyBound}
\left|\left(u - \chi_1(v - q) - \chi_2(v - q)^2\right)^{-1}v^{-1 + O(1)}\right| \leq 2\left|p^{-1}q^{-1 + O(1)}\right|.
\end{equation}
Plugging Equation \eqref{PhiBound} and Equation \eqref{EasyBound} back into Equation \eqref{integrand} gives the following:
\begin{equation} \label{ExpDecay}
\left| (u - \chi_1(v - q) - \chi_2(v - q)^2)^{-r-1} v^{-s-1} \right| \leq 2p^{-r - 1}q^{-s - 1 + O(1)} e^{-\frac{d}{2}r^{\frac{1}{5}}}.
\end{equation}
Recognizing that the entire domain of integration has size bounded by a constant, we combine Equation \eqref{HBound} and Equation \eqref{ExpDecay} to get the desired result:
\[
\left(\frac{1}{2 \pi i}\right)^2 \iint_{\mathcal{\bar C}(p, q)} \tilde H(u, v)^{-\beta} (u - \chi_1(v - q) - \chi_2(v - q)^2)^{-r - 1}v^{-s - 1} \, \D u \, \D v = O\left(p^{-r}q^{-s} r^{|\beta|} e^{-\frac{d}{2} r^{\frac{1}{5}}}\right).
\]
\end{proof}

Now, we examine the corresponding statement for the product integral.

\begin{lemma} \label{vbig2}
Let $\mathcal{C_\ell^*}$ represent the portion of $\mathcal{C_\ell}$ where $|u - q| \geq r^{-\frac{7}{10}}$.  Then, the following holds uniformly as $r, s \to \infty$ with {$\lambda = \frac{r + O(1)}{s}$} {for some constant $d > 0$}:
\begin{multline*}
\left( \frac{1}{2 \pi i}\right)^2 \iint_{\mathcal{C_\ell^*}(p, q)} [H_x(p, q) \cdot (u - p)]^{-\beta} u^{-r - 1} v^{-s - 1} \left[1 - \frac{\chi_1(v - q) + \chi_2(v - q)^2}{p} \right]^{-r - 1} \,\D u \,\D v\\
 = O\left(p^{-r}q^{-s} r^{|\beta|} e^{-\frac{d}{2} r^{\frac{1}{5}}}\right).
\end{multline*}
\end{lemma}

\begin{proof}
{This statement can be proved by using the same method as in Lemma \ref{vbig1}, by rewriting the integrand and examining the analogue of the function $\varphi(u, v)$ in Lemma \ref{vbig1}.  The corresponding function has exactly the same form as $\varphi(u, v)$ in Equation \eqref{phidef}.}
\end{proof}

We have nearly completed the proof of Lemma \ref{product}.  However, it is not yet clear that the bounds we have found away from the critical point are small compared to the value of the whole integral.  It turns out that the exponential term in these bounds, $e^{-\frac{d}{2}r^{\frac{1}{5}}}$, will ensure that these bounds are small compared to the integral overall.  To show this, it remains to evaluate the asymptotic contribution of the product integral, which will simultaneously show that the contributions to the integral away from the critical point are negligible.

\section{Analyzing the Product Integral} \label{theoremproof}
Lemma \ref{product} has reduced our work to computing the following:
\[
\left( \frac{1}{2 \pi i}\right)^2 \iint_{\mathcal{C}_\ell(p, q)} [H_x(p, q) \cdot (u - p)]^{-\beta} u^{-r - 1} v^{-s - 1} \left[1 - \frac{\chi_1(v - q) + \chi_2(v - q)^2}{p} \right]^{-r - 1} \,\D u \,\D v.
\]
We break it up into two univariate integrals, to be analyzed separately:
\begin{multline} \label{prodint}
\left( \frac{1}{2\pi i}\right)^2 \left( \int_{U} [H_x(p, q) \cdot (u - p)]^{-\beta} u^{-r - 1}\, \D u\right) \cdot \\
 \left( \int_{V} v^{-s - 1} \left[1 - \frac{\chi_1 (v - q) + \chi_2(v - q)^2}{p}\right]^{-r-1} \, \D v\right).
\end{multline}
Above, $U$ is the $u$-projection of the contour, $\mathcal{C}_\ell$, which resembles the $x$ contour in Figure \ref{fig:contour}, but with $G(y) = 0$.  $V$ is likewise the $v$-projection, which is the set, $\left\{v : v = qe^{i \theta}, |\theta| \leq r^{-\frac{2}{5}} \right\}$.  We analyze each integral in lemmas below.

\begin{lemma} \label{uint}
The following holds uniformly as $r, s \to \infty$ with $\lambda = \frac{r + O(1)}{s}$:
\[
\int_{U} [H_x(p, q) \cdot (u - p)]^{-\beta} u^{-r - 1}\, \D u
= \frac{2 \pi i}{\Gamma(\beta)} r^{\beta - 1} p^{-r} \left\{(-H_x(p, q) p)^{-\beta}\right\}_P e^{-\beta(2 \pi i \omega)} + o\left(r^{\beta - 1} p^{-r}\right).
\]
Here, $\omega$ is defined to be the signed number of times the curve $H(tp, tq)$ crosses the branch cut in the definition of the function $\left\{x^{-\beta}\right\}_P$, as described in the statement of the Theorem.
\end{lemma}

The main idea behind the proof of this lemma is that the remaining integral is almost a univariate Cauchy integral, but with a shrunken domain of integration.  Thus, the integral is approximately equal to the coefficients of $[H_x(p, q) \cdot (u - p)]^{-\beta}$, which can be estimated using the binomial theorem and Stirling's approximation. To account for the branch cut in the original function $H(x, y)^{-\beta}$, we add a term $e^{-\beta(2 \pi i \omega)}$, where $\omega$ counts the number of times the image of $H$ wraps around the origin as the input increases from $(0, 0)$ to $(p, q)$.  An example is illustrated in Figure \ref{fig:arg}.

\begin{proof}

The contour $U$ is comprised of the segments $\gamma_i$ for $1 \leq i \leq 5$ in the case where $|v - q| \leq r^{-\frac{2}{5}}$.  The endpoints of the contour, at the beginning of $\gamma_4$ and end of $\gamma_5$, both have magnitude $|u| = |p| + \epsilon_x$.  We can attach these endpoints to a portion of the circle $\{u : |u| = |p| + \epsilon_x\}$ to form a closed cycle $\mathcal{C}_u$ that wraps around the origin and contains no singularities of $[H_x(p, q) \cdot (u - p)]^{-\beta}$.  Because $u^{-r - 1}$ is exponentially smaller on the circle $\{u : |u| = |p| + \epsilon_x\}$ than it is near the critical point $p$, we have:
\[
\int_{U} [H_x(p, q) \cdot (u - p)]^{-\beta} u^{-r - 1}\, \D u = (1 + o(1)) \int_{\mathcal{C}_u} [H_x(p, q) \cdot (u - p)]^{-\beta} u^{-r - 1}\, \D u.
\]
Now, we can use the Cauchy integral formula to evaluate this integral.  However, we finally must worry about how the analytic continuation of $H^{-\beta}$ is defined.  $H(0, 0)$ is nonzero by assumption, and the values of $H^{-\beta}$ are defined near the origin of $\mathbb{C}^2$ by the generating function itself.  Separately from the analytic continuation of $H^{-\beta}$ that we have used up to this point, we choose a branch of the logarithm with the following properties: the branch must agree with $H^{-\beta}$ on some small neighborhood of the origin, and its branch cut must be a line from the origin that is \emph{not} the line $\ell(t) = -tH_x(p, q)p$ for $t \geq 0$, for any of the critical points $(p, q)$.  Define $\left\{x^{-\beta}\right\}_P$ as the value of $x^{-\beta}$ obtained by using this branch of the logarithm.

Consider the curve $H(tp, tq)$ in $\mathbb{C}$, with $t \in [0, 1)$.  This curve may wrap around the origin several times, and in particular, may cross the branch cut described above.  Recall the bivariate power series for $H(x, y)$:
\[
H(x, y) = \sum_{m, n \geq 0} h_{mn}(x - p)^m (y - q)^n.
\]

Plugging in our parameterization yields:
\begin{align*}
H(tp, tq) &= h_{10}(tp - p) + h_{01}(tq - q) + \cdots\\
&= (1 - t)(-ph_{10} - qh_{01}) + O(1 - t)^2.
\end{align*}
The above equations are true as $t \to 1$.  Recall the following conditions: $H_x(p, q) \neq 0$, and $H_y(p, q) = \frac{p}{\lambda q}H_x(p, q)$. Plugging this into our computations above yields:
\[
H(tp, tq) = (1 - t)(-p(1 + \lambda) H_x(p, q)) + O(1 - t)^2.
\]
Thus, as $t$ tends to $1$, the curve $H(tq, tq)$ is essentially linear, with quadratic error.  As long as the branch cut chosen above is not the line $\ell(t)$ mentioned above, the curve will only cross the branch cut finitely many times.  Let $\omega$ be the signed number of times the curve $H(tp, tq)$ crosses the branch cut in the counter-clockwise direction for $t \in [0, 1)$.  That is, every time the curve crosses the branch cut in the counter-clockwise direction, add $1$ to $\omega$, and every time it crosses in the clockwise direction, subtract $1$ from $\omega$.  If the curve only touches the branch cut without crossing it, leave $\omega$ unchanged.

As $t$ approaches $1$, we have shown that $H$ behaves essentially like $H_x(p, q)(u - p)$, and we have traced how the argument changes as we expand the two-dimensional torus towards the critical point.  Now, in order to revert the integral over $\mathcal{C}_u$ back to the appropriate coefficient of $H_x(p, q)(u - p)$ by using the Cauchy integral formula, we must follow the image of $H_x(p, q)(u - p)$ from $u = p$ back to the origin $u = 0$.  As $u$ follows the line from $p$ to $0$, the $H_x(p, q)(u - p)$ will follow the line in $\mathbb{C}$ from $0$ to $-pH_x(p, q)$, the point whose power we are trying to determine.  Because this straight line is $\ell(t)$, it will not cross the branch cut we chose above.  Thus, $\omega$ already accounts for the total number of times the branch cut is crossed.
Figure \ref{fig:arg} shows an example of this setup.
\begin{figure}
\centering
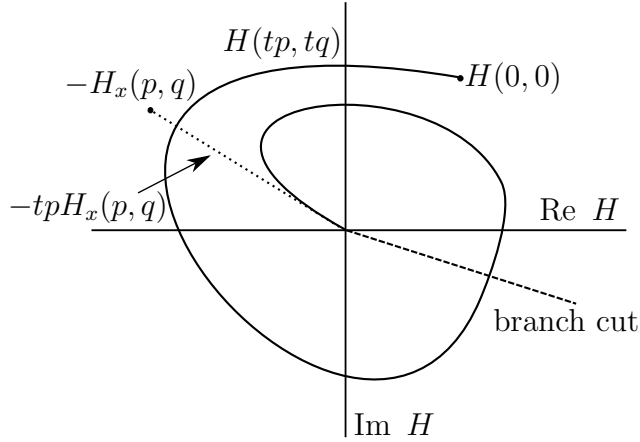
\caption{An example with $\omega = 1$.}
\label{fig:arg}
\end{figure}
In this example, $\omega = 1$, because $H(tp, tq)$ crosses the branch cut once in the counter-clockwise direction.

In conclusion, we have the following:
\begin{align*}
\int_{\mathcal{C}_u} [H_x(p, q) \cdot (u - p)]^{-\beta}& u^{-r - 1}\, \D u \\
&=  (1 + o(1)) 2 \pi i \left[u^r\right] (H_x(p, q) \cdot (u - p))^{-\beta} \nonumber \\
&= (1 + o(1)) 2 \pi i (H_x(p, q))^{r} \binom{-\beta}{r} \nonumber  \cdot \left\{(-H_x(p, q) p)^{-\beta - r}\right\}_P e^{-\beta(2 \pi i \omega)} \nonumber \\
&= {\frac{2 \pi i}{\Gamma(\beta)} r^{\beta - 1} p^{-r} \left\{(-H_x(p, q) p)^{-\beta}\right\}_P e^{-\beta(2 \pi i \omega)}  \nonumber} {+ o\left(r^{\beta - 1}p^{-r}\right).} \nonumber
\end{align*}
{In the last line, we used Stirling's approximation to complete the proof.}
\end{proof}

We turn our attention to the other integral, and find its asymptotic contribution.
\begin{lemma} \label{vintlemma}
The following holds uniformly as $r, s \to \infty$ with $\lambda = \frac{r + O(1)}{s}$:
\[
\int_{V} v^{-s-1} \left[1 - \frac{\chi_1(v-q) + \chi_2(v - q)^2}{p}\right]^{-r-1} \, \D v = iq^{-s} \sqrt{\frac{2\pi}{-q^2Mr}} + o\left(q^{-s} r^{-\frac{1}{2}}\right).
\]
Here, the square root is taken to be the principal root.
\end{lemma}
This integral is nearly a Fourier-Laplace integral, again with a shrunken domain of integration, and thus can be estimated using standard Fourier-Laplace approximations.

\begin{proof}
Note that $\lambda = \frac{r + O(1)}{s}$ implies that $s = -\frac{r}{\lambda} + O(1)$.  We rewrite the integrand:
\[
v^{-s - 1}\left[1 - \frac{\chi_1(v - q) + \chi_2(v - q)^2}{p}\right]^{-r - 1} = q^{-s - 1} \left(\frac{v}{q}\right)^{O(1)} e^{-r \psi(v)}.
\]
{Above, we define $\psi$ as $\psi(v) := \log\left[1 - \frac{\chi_1(v - q) + \chi_2(v - q)^2}{p}\right] + \frac{1}{\lambda} \log\left(\frac{v}{q}\right)$.  Next, we expand} $\psi(v)$ as a Taylor series about $v = q$:
\[
\psi(v) = \frac{M}{2}(v - q)^2 + O(v - q)^3.
\]
{Also, since $v = qe^{i\theta}$ and $|\theta| \leq r^{-\frac{2}{5}}$, we have that $\left(\frac{v}{q}\right)^{O(1)} = 1 + o(1)$}.  Plugging these expressions into the integral and rewriting it in terms of $\theta$ gives us the following:
\begin{align}
\int_{V} v^{-s-1} &\left[1 - \frac{\chi_1(v-q) + \chi_2(v - q)^2}{p} \right]^{-r-1} \, \D v  \nonumber\\
&\hspace{3 cm} = q^{-s - 1} \int_V e^{-r\left[\frac{M}{2}(v - q)^2 + O(v - q)^3\right]} (1 + o(1)) \, \D v \nonumber \\
&\hspace{3 cm} = iq^{-s} [1 + o(1)] \int_{-r^{-\frac{2}{5}}}^{r^{\frac{2}{5}}} e^{-r \left[ -\frac{q^2 M}{2}\theta^2 + O(\theta)^3\right]} e^{i\theta} \, \D \theta \nonumber\\
&\hspace{3 cm}=iq^{-s} [1 + o(1)] \int_{-r^{-\frac{2}{5}}}^{r^{\frac{2}{5}}} e^{-r \left[-\frac{q^2 M}{2}\theta^2\right]} e^{i\theta} \, \D \theta.
\label{vint}
\end{align}
The last line is true because $O(\theta)^3 = O\left(r^{-\frac{6}{5}}\right)$ implies that $e^{-rO(\theta)^3} = 1 + o(1)$.  The remaining integral is nearly a Fourier-Laplace integral, but it has a shrinking domain of integration.  We justify that this can be replaced by a domain of integration of constant size.  Specifically, we aim to show that for some $\epsilon > 0$ small enough, the following holds:
\begin{equation}
\int_{-r^{-\frac{2}{5}}}^{r^{\frac{2}{5}}} e^{-r \left[ -\frac{q^2 M}{2}\theta^2\right]} e^{i\theta} \, \D \theta  = 
\int_{-\epsilon}^\epsilon e^{-r \left[ -\frac{q^2 M}{2}\theta^2\right]} e^{i\theta} \, \D \theta + O\left(e^{-\frac{d}{2} r^{\frac{1}{5}}}\right). \label{vconstant}
\end{equation}
To see this, notice that if $|\theta| \geq r^{-\frac{2}{5}}$, then we have:
\[
\re \left( -\frac{q^2 M}{2} \theta^2 \right) \geq \frac{d}{2} r^{-\frac{4}{5}}.
\]
{Above, as before, we define $d = \re\left( -\frac{q^2 M}{2}\right)$.}   Then, we have the following:
\[
\left|e^{-r \left[-\frac{q^2 M}{2} \theta^2\right]}\right| \leq e^{-\frac{d}{2} r^{\frac{1}{5}}}.
\]
This justifies Equation \eqref{vconstant}.  To analyze the remaining integral, we use the standard saddle point approximation, which is proved in Theorem 4.1.1 in \cite{PW2013}. The amplitude $A(\theta) = e^{i\theta}$ and the phase $\phi(\theta) = -\frac{q^2 M}{2} \theta^2$ are both analytic functions near $\theta = 0$, and $\re(\phi) \geq 0$ on the interval $[-\epsilon, \epsilon]$, with equality only at $\theta = 0$.  Thus, we have:
\begin{equation}
\int_{-\epsilon}^\epsilon e^{-r \left[ -\frac{q^2 M}{2}\theta^2\right]} e^{i\theta} \, \D \theta = (1 + o(1)) A(0) \sqrt{\frac{2 \pi}{\phi''(0)r}} e^{-r \phi(0)} = (1 + o(1)) \sqrt{\frac{2\pi}{-q^2 Mr}}. \label{FLint}
\end{equation}
In the above expression, the square root is the principal root.  Plugging Equation \eqref{FLint} into the remaining integral in Equation \eqref{vint} finishes the proof.
\end{proof}
Plugging the results of Lemma \ref{uint} and Lemma \ref{vintlemma} into Equation \eqref{prodint} gives us the final result:
\begin{multline*}
\left( \frac{1}{2\pi i}\right)^2 \int_{U} [H_x(p, q) \cdot (u - p)]^{-\beta} u^{-r - 1}\, \D u \int_{V} v^{-s - 1} \left[1 - \frac{\chi_1}{p}(v - q) - \frac{\chi_2}{p}(v - q)^2\right]^{-r-1} \, \D v
\\ = [1 + o(1)] \frac{r^{\beta - \frac{3}{2}} p^{-r} q^{-s} \left\{(-H_x(p, q)p)^{-\beta}\right\}_P e^{-\beta(2 \pi i \omega)}}{\Gamma(\beta) \sqrt{-2 \pi q^2 M}}.
\end{multline*}

Unfortunately, for general $H$, the formula in the Theorem becomes quite messy, as we must find how many times the image of $H$ wraps around the origin along the path connecting $(0, 0)$ to each critical point $(p, q)$.  Additionally, the sign of the square root in the formula can cause headaches.  Luckily, in the case where $H$ has only real coefficients and there is a single smooth strictly minimal critical point, we can simplify the formula, as seen in Corollary \ref{GrahamCorollary} above.  The proof of the Corollary is simple from here. 

\begin{proof}
 Since $H$ has real coefficients and $p$ and $q$ are positive real numbers, we must have that $-H_x(p, q)p$ is real.  The line $H(tp, tq)$ for $0 \leq t \leq 1$ is real and can't pass through the origin since $(p, q)$ is minimal.  Also, the line from $0$ to $-H_x(p, q) p$ is real, and it approximates $H(tp, tq)$ for $t$ near $1$, as described in Section \ref{theoremproof} and Figure \ref{fig:arg} above.  This would mean that $-H_x(p, q) p$ is in fact positive.  Additionally, the line $H(tp, tq)$ cannot wrap around the origin, which forces $\omega = 0$ in the statement of the original Theorem.  As a result, $\{(-H_x(p, q)p)^{-\beta}\}_P$ is positive.  With this term positive, the only other term with an unknown sign is $\sqrt{-2 \pi q^2 M}$.  However, knowing that $-2 \pi q^2 M$ is real, in order for the coefficients of $H^{-\beta}$ to be real at all, $-2 \pi q^2 M$ must be negative. Thus, $\sqrt{-2 \pi q^2 M}$ is always positive (since the principal root is taken), which forces the whole formula to be positive always.
 \end{proof}

\section{Examples} \label{Examples}
\subsection{{A Generating Function with Multinomial Coefficients}}
{In order to see how Theorem \ref{MainResult} can be applied, we look at an artificial example where the computations are particularly simple, and where it is easy to verify that there is a strictly minimal critical point.  Consider the coefficients $x^ry^r$ of the generating function,}
  \[
  F(x, y) = \frac{1}{(1 - x - y)^{\frac{1}{2}}}.
  \]
  {Using the Multinomial Theorem, we have that the coefficient $[x^ry^s]F(x, y) = \binom{-\frac{1}{2}}{r} \binom{-\frac{1}{2} - r}{s}$.  Thus, using Theorem \ref{MainResult}, we will find an approximation for $\binom{-\frac{1}{2}}{r} \binom{-\frac{1}{2} - r}
  {r}$ for large $r$.}  Here, the direction $\lambda = 1$.  So, the critical point equations become the following:
  \begin{align*}
  1 - x - y &= 0,\\
  x &= y.
  \end{align*}
  The unique solution to these equations is $(p, q) = (1/2, 1/2)$.  Because the $x$ partial derivative of $1 - x - y$ never vanishes, this point is a smooth critical point.  Additionally, it is easy to verify that $(1/2, 1/2)$ must be a strictly minimal critical point: on the zero set of $1 - x - y$, we have that $y = 1 - x$, and it is clear that if $|x| \leq \frac{1}{2}$, then $|y| \leq \frac{1}{2}$, with equality only when $x = y = \frac{1}{2}$.  Thus, we can apply Corollary \ref{GrahamCorollary} to approximate $[x^ry^{r}]F(x, y)$.  With $H(x, y) = 1 - x - y$ and $\lambda = 1$, we compute the following:
  \[
  \chi_1 = 1, \chi_2 = 0, H_x(1/2, 1/2) = -1, M = -8.
  \]
  Plugging these into Corollary $1$ with $\beta = 1/2$ yields:
  \[
  [x^ry^{r}]F(x, y) \sim \frac{2^{2r - 1/2}}{\pi r}.
  \]
  For $r = {1}00$, this approximation yields $3.61688 \cdot 10^{57}$, while the actual coefficient of $F(x, y)$, ${\binom{-1/2}{100}\binom{-1/2 - 100}{100}}$, is approximately $3.61011 \cdot 10^{57}$.
  
\subsection{Complete Graph Coloring}
 \label{Example}
Using Corollary \ref{GrahamCorollary} above, we will find an asymptotic formula for the coefficients of the following bivariate generating function:
\[
F(x, y) =  \frac{1-x(1+y)}{\sqrt{1-2x(1+y) -x^2(1-y)^2}}.
\]
This generating function describes the stationary distributions of a red/blue color-swapping algorithm on the complete graphs $K_r$, where the coefficient of $x^ry^s$ is proportional to the probability that $s$ of the vertices of $K_r$ are blue in the stationary distribution, rescaled by a factor proportional to $\binom{2r - 2}{r}$.  For more details, see \cite{GG2015}.  Because each $y$ term is attached to an $x$ of equal or greater power, the power series expansion of $F$ will have no terms where the power of $y$ is larger than the power of $x$.  Thus, we will look at the asymptotics only the case where $\mu := \lambda^{-1} = \frac{s}{r} \in (0, 1)$.  (We switch to $\mu$ so that the range of possible directions is bounded.)

The following paragraphs briefly describe how to check the conditions of Corollary \ref{GrahamCorollary} computationally.  A Maple worksheet providing the code for these computations is available online:
\begin{center}
\url{https://github.com/TorinGreenwood/JCTA-BivAlg}
\end{center}
To begin, we find the critical points of the denominator, $H(x, y) = 1-2x(1+y) -x^2(1-y)^2$.  We can use a Gr\"{o}bner basis to compute these points in terms of $\mu$. 
Here, the first polynomial in the basis is as follows:
\begin{equation*} \label{xpoly}
1-2\mu+\mu^2+(-4-2\mu^2+6\mu)x+2x^3+(2\mu^2-4\mu+3)x^2.
\end{equation*}
Because this is a degree $3$ polynomial in $x$, we can solve for the three values of $x$ explicitly in terms of $\mu$.
Once the $x$ solutions are found, they can be plugged into the second basis element of the Gr\"{o}bner basis to compute the corresponding $y$ solutions in terms of $\mu$.  We must check that all of these critical points are smooth, that neither $x$ nor $y$ is zero, and that $M$ is not zero at each critical point where $\mu \in (0, 1)$.  Each of these facts can be verified with additional Gr\"{o}bner bases.   Showing that $M$ is nonzero requires a slightly more complicated Gr\"{o}bner basis, but it is easy to verify via Sturm sequences that $M$ is never zero for $\mu \in (0, 1)$.  Finally, checking the critical points for minimality is computationally difficult, but can be achieved for any particular $\mu$ through quantifier elimination, using Mohab Safey El Din's RAGlib package, \cite{S2015}, which relies on the FGb Maple package, \cite{F2010}. 
To see how closely the formula approximates coefficients, we look at the case where $\mu = \frac{1}{2}$.  Using the Gr\"{o}bner bases mentioned above, we easily compute the unique minimal critical point, $(p, q) = \left(\frac{1}{4}, 1\right)$.  From here, we compute the following:
\[
\chi_1 = \frac{1}{8}, \ \ \chi_2 = -\frac{3}{64}, \ \ H_x\left(\frac{1}{4}, 1\right) = -4, \ \ M = -\frac{3}{8}.
\]
Thus, from Corollary \ref{GrahamCorollary} above (with $\beta = \frac{1}{2}$), we have that as $r, s \to \infty$ with $2 = \frac{r + O(1)}{s}$ as $r, s \to \infty$,
\[
\left[x^ry^s\right] H(x, y)^{-1/2} \sim \frac{r^{-1} \left(\frac{1}{4}\right)^{-r}}{\Gamma(\frac{1}{2}) \sqrt{\frac{3\pi}{4}}} = \frac{2 \cdot 4^r}{r \pi \sqrt{3}}.
\]
If the numerator of $F$ was a monomial $ax^my^n$, it would simply shift the terms in the series of $F$ by $m$ in the $x$ variable and $n$ in the $y$ variable, and multiply all the coefficients by $a$.  We can break up the numerator of $F$ linearly and compute these shifts separately.  Equivalently, to account for the fact that the numerator $G(x, y) :=  1 - x(1 + y)$ is not a monomial, we multiply our approximation above by $G$ evaluated at the critical point.  In this case, since $G\left(\frac{1}{4}, 1\right) = \frac{1}{2}$, the final approximation is:
\[
[x^r y^s] F(x, y) \sim \frac{4^r}{r \pi \sqrt{3}}.
\]
When $r = 70,$ this formula gives approximately $3.65924 \cdot 10^{39}$.  The actual value of $\left[x^{70}y^{35}\right] F(x, y)$ is approximately $3.59821 \cdot 10^{39}$.  The ratio of these values is $1.017$, showing that the approximation is already quite good for $r = 70$.

\section{Future Research}

Possible future research directions include finding more complete asymptotic expansions for the coefficients of $H^{-\beta}$, and also extending to more variables.  Finding a more complete asymptotic expansion should be possible by analyzing the Cauchy integral using the same contour, but with more precise error handling.  Extending the formula to more variables is challenging because the change of variables needed to approximate a multivariate function by a univariate linear function quickly becomes complicated.  

Another research direction would be to look at other types of algebraic singularities.  For example, one could study the coefficients of a function $F(x, y)$ which is known to satisfy some polynomial equation but may not have the form $H^{-\beta}$. {One challenge in this case is to identify which singularities closest to the origin contribute to the asymptotics of the coefficients.  If it is possible to determine the singularities, then the equation $F$ satisfies can be used to estimate the behavior of $F$ near the singularities.  Assuming $F$ is reasonably well-behaved near its singularities, the results in this paper could be applied like a transfer theorem to compute the asymptotics for the coefficients of $F$.  In some cases, it is possible to find accurate computational estimates for the coefficients of $F$, which may be enough to determine where the contributing critical points for $F$ are located.
}

Combining these results with other asymptotic techniques may yield stronger results and more complete asymptotic expansions, too.  For example, creative telescoping methods take the generating function in question and find a partial differential equation that the function satisfies.  By finding a basis of solutions to this differential equation, one can find complete asymptotic expansions to the coefficients of the generating function.  Unfortunately, it is often difficult to find the correct coefficients of the solution to the PDE -- this is referred to as the \emph{connection problem}.  However, if the leading-term asymptotics of the solution are known, the connection problem can often be solved.  Thus, combining these creative telescoping methods with the first-order asymptotics results in this paper, one may be able to analyze generating functions without too many technical computations.

\subsection*{Acknowledgements}
\label{sec:ack}
The University of Pennsylvania supported this research in part, as I completed my doctoral thesis.  This work was also supported in part by NSF grant DMS-1344199.  I thank my advisor, Robin Pemantle, for his generous help and guidance.  Mohab Safey El Din and Rainer Sinn graciously helped significantly with computations in the example.  Thank you to the referees for the 28th International Conference on Formal Power Series and Algebraic Combinatorics for their valuable feedback and references, and to the anonymous referees for the Journal of Combinatorial Theory, Series A, for numerous very helpful suggestions.

\nocite{*}
\bibliographystyle{alpha}
\bibliography{sample}
\label{sec:biblio}

\end{document}

%% file: FOHankel.pdf_tex
\begingroup%
  \makeatletter%
  \providecommand\color[2][]{%
    \errmessage{(Inkscape) Color is used for the text in Inkscape, but the package 'color.sty' is not loaded}%
    \renewcommand\color[2][]{}%
  }%
  \providecommand\transparent[1]{%
    \errmessage{(Inkscape) Transparency is used (non-zero) for the text in Inkscape, but the package 'transparent.sty' is not loaded}%
    \renewcommand\transparent[1]{}%
  }%
  \providecommand\rotatebox[2]{#2}%
  \ifx\svgwidth\undefined%
    \setlength{\unitlength}{194.11537908bp}%
    \ifx\svgscale\undefined%
      \relax%
    \else%
      \setlength{\unitlength}{\unitlength * \real{\svgscale}}%
    \fi%
  \else%
    \setlength{\unitlength}{\svgwidth}%
  \fi%
  \global\let\svgwidth\undefined%
  \global\let\svgscale\undefined%
  \makeatother%
  \begin{picture}(1,0.99429488)%
    \put(0,0){\includegraphics[width=\unitlength]{FOHankel.pdf}}%
    \put(0.72201283,0.46302134){\color[rgb]{0,0,0}\makebox(0,0)[lb]{\smash{$1$}}}%
    \put(0.88463289,0.45251796){\color[rgb]{0,0,0}\makebox(0,0)[lb]{\smash{$1/n$}}}%
  \end{picture}%
\endgroup%

%% file: ylocalcontour.pdf_tex
\begingroup%
  \makeatletter%
  \providecommand\color[2][]{%
    \errmessage{(Inkscape) Color is used for the text in Inkscape, but the package 'color.sty' is not loaded}%
    \renewcommand\color[2][]{}%
  }%
  \providecommand\transparent[1]{%
    \errmessage{(Inkscape) Transparency is used (non-zero) for the text in Inkscape, but the package 'transparent.sty' is not loaded}%
    \renewcommand\transparent[1]{}%
  }%
  \providecommand\rotatebox[2]{#2}%
  \ifx\svgwidth\undefined%
    \setlength{\unitlength}{218.53398438bp}%
    \ifx\svgscale\undefined%
      \relax%
    \else%
      \setlength{\unitlength}{\unitlength * \real{\svgscale}}%
    \fi%
  \else%
    \setlength{\unitlength}{\svgwidth}%
  \fi%
  \global\let\svgwidth\undefined%
  \global\let\svgscale\undefined%
  \makeatother%
  \begin{picture}(1,0.86144398)%
    \put(0,0){\includegraphics[width=\unitlength]{ylocalcontour.pdf}}%
    \put(0.20748346,0.56785009){\color[rgb]{0,0,0}\makebox(0,0)[lb]{\smash{$|q|$}}}%
    \put(0.59946329,0.71761601){\color[rgb]{0,0,0}\makebox(0,0)[lb]{\smash{$q$}}}%
    \put(0.81429195,0.46061032){\color[rgb]{0,0,0}\makebox(0,0)[lb]{\smash{$\re \ y$}}}%
    \put(0.45154635,0.00761466){\color[rgb]{0,0,0}\makebox(0,0)[lb]{\smash{$\II \ y$}}}%
    \put(0.50033649,0.4714244){\color[rgb]{0,0,0}\makebox(0,0)[lb]{\smash{$\theta_y$}}}%
    \put(0.68338377,0.72565662){\color[rgb]{0,0,0}\makebox(0,0)[lb]{\smash{$\mathcal{C}_y$}}}%
  \end{picture}%
\endgroup%

%% file: xlocalcontourcloseup.pdf_tex
\begingroup%
  \makeatletter%
  \providecommand\color[2][]{%
    \errmessage{(Inkscape) Color is used for the text in Inkscape, but the package 'color.sty' is not loaded}%
    \renewcommand\color[2][]{}%
  }%
  \providecommand\transparent[1]{%
    \errmessage{(Inkscape) Transparency is used (non-zero) for the text in Inkscape, but the package 'transparent.sty' is not loaded}%
    \renewcommand\transparent[1]{}%
  }%
  \providecommand\rotatebox[2]{#2}%
  \ifx\svgwidth\undefined%
    \setlength{\unitlength}{192.9046875bp}%
    \ifx\svgscale\undefined%
      \relax%
    \else%
      \setlength{\unitlength}{\unitlength * \real{\svgscale}}%
    \fi%
  \else%
    \setlength{\unitlength}{\svgwidth}%
  \fi%
  \global\let\svgwidth\undefined%
  \global\let\svgscale\undefined%
  \makeatother%
  \begin{picture}(1,0.76100317)%
    \put(0,0){\includegraphics[width=\unitlength]{xlocalcontourcloseup.pdf}}%
    \put(0.19908748,0.09408509){\color[rgb]{0,0,0}\makebox(0,0)[lb]{\smash{$p + G(y)$}}}%
    \put(0.21466755,0.205262){\color[rgb]{0,0,0}\makebox(0,0)[lb]{\smash{$1/r$}}}%
    \put(0.42211863,0.03777705){\color[rgb]{0,0,0}\makebox(0,0)[lb]{\smash{$\gamma_1$}}}%
    \put(0.54125902,0.3309389){\color[rgb]{0,0,0}\makebox(0,0)[lb]{\smash{$\gamma_2$}}}%
    \put(0.44306485,0.42737533){\color[rgb]{0,0,0}\makebox(0,0)[lb]{\smash{$\gamma_3$}}}%
    \put(0.55617389,0.62636374){\color[rgb]{0,0,0}\makebox(0,0)[lb]{\smash{$\gamma_5$}}}%
    \put(0.74259471,0.47974071){\color[rgb]{0,0,0}\makebox(0,0)[lb]{\smash{$\gamma_4$}}}%
  \end{picture}%
\endgroup%

%% file: theta2.pdf_tex
\begingroup%
  \makeatletter%
  \providecommand\color[2][]{%
    \errmessage{(Inkscape) Color is used for the text in Inkscape, but the package 'color.sty' is not loaded}%
    \renewcommand\color[2][]{}%
  }%
  \providecommand\transparent[1]{%
    \errmessage{(Inkscape) Transparency is used (non-zero) for the text in Inkscape, but the package 'transparent.sty' is not loaded}%
    \renewcommand\transparent[1]{}%
  }%
  \providecommand\rotatebox[2]{#2}%
  \ifx\svgwidth\undefined%
    \setlength{\unitlength}{261.94812012bp}%
    \ifx\svgscale\undefined%
      \relax%
    \else%
      \setlength{\unitlength}{\unitlength * \real{\svgscale}}%
    \fi%
  \else%
    \setlength{\unitlength}{\svgwidth}%
  \fi%
  \global\let\svgwidth\undefined%
  \global\let\svgscale\undefined%
  \makeatother%
  \begin{picture}(1,0.54972717)%
    \put(0,0){\includegraphics[width=\unitlength]{theta2.pdf}}%
    \put(0.15523103,0.26395631){\color[rgb]{0,0,0}\makebox(0,0)[lb]{\smash{$\theta_x$}}}%
    \put(0.71412032,0.30278622){\color[rgb]{0,0,0}\makebox(0,0)[lb]{\smash{$\frac{u}{p}$}}}%
    \put(0.57341639,0.2307981){\color[rgb]{0,0,0}\makebox(0,0)[lb]{\smash{$\alpha$}}}%
    \put(0.01854242,0.16762208){\color[rgb]{0,0,0}\makebox(0,0)[lb]{\smash{$\frac{\tilde \kappa(v)}{p}$}}}%
    \put(0.39048092,0.16980361){\color[rgb]{0,0,0}\makebox(0,0)[lb]{\smash{$1 + \frac{\tilde \kappa(v)}{p}$}}}%
    \put(0.62389678,0.18289223){\color[rgb]{0,0,0}\makebox(0,0)[lb]{\smash{$\frac{\epsilon_x}{|p|}$}}}%
    \put(0.31194834,0.10021502){\color[rgb]{0,0,0}\makebox(0,0)[lb]{\smash{$1 + \frac{\epsilon_x}{|p|}$}}}%
  \end{picture}%
\endgroup%

%% file: argumentargument2.pdf_tex
\begingroup%
  \makeatletter%
  \providecommand\color[2][]{%
    \errmessage{(Inkscape) Color is used for the text in Inkscape, but the package 'color.sty' is not loaded}%
    \renewcommand\color[2][]{}%
  }%
  \providecommand\transparent[1]{%
    \errmessage{(Inkscape) Transparency is used (non-zero) for the text in Inkscape, but the package 'transparent.sty' is not loaded}%
    \renewcommand\transparent[1]{}%
  }%
  \providecommand\rotatebox[2]{#2}%
  \ifx\svgwidth\undefined%
    \setlength{\unitlength}{269.84691162bp}%
    \ifx\svgscale\undefined%
      \relax%
    \else%
      \setlength{\unitlength}{\unitlength * \real{\svgscale}}%
    \fi%
  \else%
    \setlength{\unitlength}{\svgwidth}%
  \fi%
  \global\let\svgwidth\undefined%
  \global\let\svgscale\undefined%
  \makeatother%
  \begin{picture}(1,0.61643739)%
    \put(0,0){\includegraphics[width=\unitlength]{argumentargument2.pdf}}%
    \put(0.73940375,0.30523672){\color[rgb]{0,0,0}\makebox(0,0)[lb]{\smash{$\re \ H$}}}%
    \put(0.48023101,0.00435722){\color[rgb]{0,0,0}\makebox(0,0)[lb]{\smash{$\II \ H$}}}%
    \put(0.67509066,0.14934959){\color[rgb]{0,0,0}\makebox(0,0)[lb]{\smash{$\mbox{branch cut}$}}}%
    \put(0.63498175,0.49149881){\color[rgb]{0,0,0}\makebox(0,0)[lb]{\smash{$H(0, 0)$}}}%
    \put(0.29734723,0.54062104){\color[rgb]{0,0,0}\makebox(0,0)[lb]{\smash{$H(tp, tq)$}}}%
    \put(0.07626336,0.48279198){\color[rgb]{0,0,0}\makebox(0,0)[lb]{\smash{$-H_x(p, q)$}}}%
    \put(-0.00246089,0.31147367){\color[rgb]{0,0,0}\makebox(0,0)[lb]{\smash{$-tpH_x(p, q)$}}}%
  \end{picture}%
\endgroup%

%% file: UpdatedArXiV.bbl
\begin{thebibliography}{99999999}

\bibitem[BCCG15]{GG2015} S. Butler, F. Chung, J. Cummings, and R. Graham.  \emph{Edge flipping in the complete graph.} Preprint.  Available online: \url{http://www.math.ucsd.edu/~ronspubs/pre_flipping.pdf}

\bibitem[BR83]{BR1983} E.A. Bender and L.B. Richmond.  \emph{Central and local limit theorems applied to asymptotic enumeration.  II.  Multivariate generating functions.}  In J. Combin. Theory Ser. A, vol 34 (1983) pp. 255-265.

\bibitem[Fau10]{F2010} J.-C. Faug\`{e}re. \emph{FGb: A Library for Computing Gr\"{o}bner Bases.}  In Komei Fukuda, Joris Hoeven, Michael Joswig, and Nobuki Takayama, editors, \emph{Mathematical Software ICMS 2010}, vol 6327 of \emph{Lecture Notes in Computer Science}, pp. 84-87, Berlin, Heidelberg, September 2010.  Available online: \url{http://www-polsys.lip6.fr/~jcf/FGb/FGb/index.html}

\bibitem[FO90]{FO1990} P. Flajolet and A. M. Odlyzko.  \emph{Singularity analysis of generating functions.}  In SIAM J. Discrete Math., vol 3 (1990) pp. 216-240. Available online: \url{http://algo.inria.fr/flajolet/Publications/FlOd90b.pdf}

\bibitem[GR92]{GR1992} Z. Gao and L.B. Richmond. \emph{Central and local limit theorems applied to asymptotic enumeration.  IV.  Multivariate generating functions.}  In J. Comput. Appl. Math., vol 41 (1992) pp. 177-186.

\bibitem[GM88]{GM1988} M. Goresky and R. MacPherson.  \emph{Stratified Morse Theory,} Springer-Verlag, Berlin Heidelberg, 1988.  Available online: \url{http://www.math.ias.edu/~goresky/MathPubl.html}

\bibitem[Gre16]{Gre2016} T. Greenwood. \emph{Asymptotics of Bivariate Analytic Functions with Algebraic Singularities.}  Extended abstract.  In DMTCS proc., BC (2016) pp. 599-610.  Available online: \url{https://fpsac2016.sciencesconf.org/browse/author?authorid=413143}

\bibitem[Hwa96]{H1996} Hwang, H.-K. \emph{Large deviations for combinatorial distributions.  I.  Central limit theorems.} In Ann. Appl. Probab., vol 6 (1996) pp. 297-319. Available online: \url{http://projecteuclid.org/download/pdf_1/euclid.aoap/1034968075}

\bibitem[Hwa98]{H1998} Hwang, H.-K. \emph{Large deviations for combinatorial distributions.  II.  Local limit theorems.} In Ann. Appl. Probab., vol 8 (1998) pp. 163-181. Available online: \url{http://projecteuclid.org/download/pdf_1/euclid.aoap/1027961038}

{
\bibitem[PH14]{PH2014} S. Poznanovi\'{c} and C.E. Heitsch.  \emph{Asymptotic distribution of motifs in a stochastic context-free grammar model of RNA folding.}  In J Math Biol., Dec 2014 pp. 1743-72.  Available online: \url{https://www.ncbi.nlm.nih.gov/pubmed/24384698}
}

\bibitem[PW13]{PW2013} R. Pemantle and M. Wilson. \emph{Analytic Combinatorics in Several Variables,} Cambridge University Press, New York, 2013. Available online: \url{http://www.math.upenn.edu/~pemantle/ACSV.pdf}

\bibitem[RW08]{RW2008} A. Raichev and M. Wilson.  \emph{Asymptotics of Coefficients of Multivariate Generating Functions: Improvements for Smooth Points.}  In Electron. J. Combin.,  vol 15 (2008).  Available online: \url{http://www.combinatorics.org/ojs/index.php/eljc/article/view/v15i1r89/pdf}

\bibitem[Saf15]{S2015} M. Safey El Din.  \emph{RAGlib, a Maple package dedicated to real solutions of polynomial systems}, March 2015.  Available online: \url{http://www-polsys.lip6.fr/~safey/RAGLib/}



\end{thebibliography}
